\documentclass[10pt]{amsart}
\usepackage{amsmath, amsthm, amssymb, amsfonts, verbatim}
\usepackage{array}
\usepackage[bookmarks]{hyperref}
\usepackage[active]{srcltx}
\usepackage{amscd}
\usepackage{epsfig}
\usepackage{color}
\usepackage{comment}
\usepackage{setspace}
\usepackage{textcomp}
\usepackage{stmaryrd}
\usepackage{multirow}

\numberwithin{equation}{section}

\newcommand{\pmat}[1]{\begin{pmatrix} #1 \end{pmatrix}}
\newcommand{\case}[1]{\begin{cases} #1 \end{cases}}

\newcommand{\beq} {\begin{equation}}
\newcommand{\eeq} {\end{equation}}
\newcommand{\bdm} {\begin{displaymath}}
\newcommand{\edm} {\end{displaymath}}

\newcommand{\bit}{\begin{itemize}}
\newcommand{\eit}{\end{itemize}}
\newcommand{\bde}{\begin{description}}
\newcommand{\ede}{\end{description}}

\newcommand{\ben}{\begin{enumerate}}
\newcommand{\een}{\end{enumerate}}
\newcommand{\algn}[1]{\begin{align} #1 \end{align}}
\newcommand{\algns}[1]{\begin{align*} #1 \end{align*}}
\newcommand{\mltln}[1]{\begin{multline} #1 \end{multline}}
\newcommand{\mltlns}[1]{\begin{multline*} #1 \end{multline*}}

\newcommand{\barr}{\begin{array}}
\newcommand{\earr}{\end{array}}
\newcommand{\subeqns}[2]{\begin{subequations} \label{#1} \algn{#2} \end{subequations}}

\newcommand{\half} {\ensuremath{\frac{1}{2}}}
\newcommand{\mc}[1]{\mathcal{#1}}
\newcommand{\LRp}[1]{\left( #1 \right)}
\newcommand{\LRb}[1]{\left \{ #1 \right \}}
\newcommand{\LRs}[1]{\left[ #1 \right]}
\newcommand{\LRa}[1]{\left< #1 \right>}

\newcommand{\LRc}[1]{\left\{ #1 \right\}}
\newcommand{\jump}[1] {\ensuremath{\left[\![{#1}]\!\right]}}

\newcommand{\ra}{\rightarrow}

\newcommand{\e}{\epsilon}

\newcommand{\R}{{\mathbb R}}

\newcommand{\X}{{\mathbb X}}
\newcommand{\I}{\Bbb{I}}

\newcommand{\calT}{\mathcal{T}}

\renewcommand{\div}{\operatorname{div}}

\newcommand{\esssup}{\operatorname{esssup}}

\newcommand{\bs}{\boldsymbol}

\newcommand{\lap}{\Delta}
\newcommand{\pd}{\partial}
\newcommand{\nor}[1]{\left\Vert#1\right\Vert}
\newcommand{\norw}[2]{\left\Vert#1\right\Vert_{#2}}

\newtheorem{theorem}{Theorem}[section]

\newtheorem{cor}[theorem]{Corollary}

\newcommand{\deuh}{\dot{e}_{\bs{u}}^h}

\newcommand{\depth}{\dot{e}_{\pt}^h}
\newcommand{\depti}{\dot{e}_{\pt}^I}

\newcommand{\depfi}{\dot{e}_{\pf}^I}
\newcommand{\depfh}{\dot{e}_{\pf}^h}

\newcommand{\dub}{\dot{\bs{u}}}
\newcommand{\dubh}{\dot{\ub}_h}
\newcommand{\dpt}{\dot{p}_{t}}
\newcommand{\dpth}{\dot{p}_{t,h}}
\newcommand{\dpf}{\dot{p}_{p}}
\newcommand{\dpfh}{\dot{p}_{p,h}}
\newcommand{\ddt}{\frac{d}{dt}}

\newcommand{\ept}{e_{\pt}}
\newcommand{\epti}{e_{\pt}^I}
\newcommand{\epth}{e_{\pt}^h}

\newcommand{\epf}{e_{\pf}}
\newcommand{\epfi}{e_{\pf}^I}
\newcommand{\epfh}{e_{\pf}^h}

\newcommand{\eu}{e_{\bs{u}}}
\newcommand{\eui}{e_{\bs{u}}^I}
\newcommand{\euh}{e_{\bs{u}}^h}

\newcommand{\fb}{\bs{f}}

\newcommand{\kapb}{\ul{\bs{\kappa}}}

\newcommand{\LtH}[3]{L^{#1}(0,t;H_{#2}^{#3})}
\newcommand{\LtL}[3]{L^{#1}(0,t;L_{#2}^{#3})}
\newcommand{\LtV}[2]{L^{#1}(0,t;#2)}

\newcommand{\pt}{p_{t}}
\newcommand{\pth}{p_{t,h}}

\newcommand{\pf}{p_{p}}
\newcommand{\pfh}{p_{p,h}}

\newcommand{\Qt}{Q_{t}}
\newcommand{\Qf}{Q_{p}}
\newcommand{\Qth}{Q_{t,h}}
\newcommand{\Qfh}{Q_{p,h}}
\newcommand{\qt}{q_{t}}
\newcommand{\qf}{q_{p}}

\newcommand{\ub}{\bs{u}}
\newcommand{\ubh}{\bs{u}_{h}}

\newcommand{\ul}{\underline}

\newcommand{\vb}{\bs{v}}
\newcommand{\Vb}{\bs{V}}
\newcommand{\Vbh}{\bs{V}_{h}}

\newcommand{\wb}{\bs{w}}

\newcommand{\Xh}{\mathcal{X}_h}

\begin{document}

\title[poroelasticity]{Analysis and preconditioning of parameter-robust finite element methods for Biot's consolidation model }
\author{Jeonghun J. Lee}
\maketitle

\begin{abstract}
In this paper we consider a three-field formulation of the Biot model which has the displacement, the total pressure, and the pore pressure as unknowns.
For parameter-robust stability analysis, we first show a priori estimates of the continuous problem with parameter-dependent norms. 
Then we study finite element discretizations which provide parameter-robust error estimates and preconditioners.
For finite element discretizations we consider standard mixed finite element as well as stabilized methods for the Stokes equations, 
and the complete error analysis of semidiscrete solutions is given. Abstract forms of parameter-robust preconditioners are investigated 
by the operator preconditioning approach.
The theoretical results are illustrated with numerical experiments.

\end{abstract}

\section{Introduction}

In poroelastic media saturated by fluids, the behaviors of porous medium and the saturating fluid flow are described by Biot's consolidation model \cite{MR0066874}. 
Poroelasticity models are widely used in geophysics and petrolium engineering applications, so development of finite element methods for the poroelastic models began more than four decades ago \cite{VermeerVerruijt1981,ZienkiewiczShiomi1984} and is still an active research area
\cite{MuradLoula1992,MuradLoula1994,MuradThomeeLoula1996,KorsaweStarke2005,PhillipsWheeler2007a,
PhillipsWheeler2007b,Yi2013,ChenLuoFeng2013,LeeEtAl2017,berger2015stabilized,NabilRiviere2018,bause2017space}. 

Poroelasticity models for practical applications have various different ranges of parameters.
For example, geophysics materials are compressible solids whereas most soft biological tissues are modelled as incompressible or nearly incompressible materials. 
It turns out that the different parameter ranges are intimately related to accuracy of numerical methods and construction of efficient iterative solvers. 
Therefore, one of main interests of numerical methods for the Biot model is robustness for model parameter ranges, 
and there are various recent studies for parameter-robust numerical methods \cite{Lee2016,Hong-Kraus,Feng-Ge-Li,Fu,Lee2018,OyarzuaRuizBaier2016} and efficient solvers \cite{hu2017nonconforming,rodrigo2017new,LeeEtAl2017,BaerlandLeeMardalWinther2017}.
Recently, a new three-field formulation for the Biot model was independently introduced in \cite{LeeEtAl2017} and \cite{OyarzuaRuizBaier2016} with different foci of interests. In \cite{LeeEtAl2017}, the main interest is construction of preconditioners robust for various parameters (large bulk and shear moduli, small hydraulic conductivity, and small time step sizes). In \cite{OyarzuaRuizBaier2016}, the main interest is optimal error estimates robust for large bulk modulus. 
The two main purposes of this work is to provide comprehensive a priori error analysis of time dependent solutions of the three-field formulation with extension to stabilized numerical methods. 
In \cite{OyarzuaRuizBaier2016}, stability of the static system is proved using compactness of a linear operator and error estimates are obtained with standard argument
but complete error analysis for time dependent problems was not given.
In contrast, we do not use the compactness argument because it is difficult to extend the error estimates to time dependent solutions. Instead, we utilize an improved energy-type estimates, and prove the a priori error estimates of time dependent solutions without using Gronwall inequality. 
We also consider stabilized methods in this paper and provide complete error analysis and an abstract form of parameter-robust preconditioners.


The paper is organized as follows. In Section~2, we introduce preliminary materials including notations, definitions, and the variational formulation of the Biot model. In Section~3, we discuss stability of the system and prove energy-type estimates of solutions. In Section~4, we discuss finite element discretizations and the a priori error estimates of semidiscrete solutions.
In Section~5, we prove stability of static system with respect to parameter-dependent norms and propose abstract forms of parameter-robust preconditioners. Finally, we present numerical results illustrating convergence of erros and parameter-robust performances of preconditioners in Section~6.

\section{Preliminaries} \label{sec:prelim}

\subsection{Notations}
Let $\Omega$ be a bounded polygonal domain with Lipschitz continuous boundary in $\R^n$ with $n=2$ or $3$. 

For a nonnegative integer $m$, $H^m (\Omega)$, $H^m(\Omega; \R^n)$ denote the standard $\R$ and $\R^n$-valued Sobolev spaces based on $L^2$ norm. 
For a Banach space $\mc{X}$ and $(a, b) \subset \R$, $C^0 (a, b; \mc{X})$ denotes the set of functions $f : (a, b) \ra \mc{X}$ which are continuous in $t \in (a,b)$. For an integer $m \geq 1$ we define 
\begin{align*}
C^m (a, b ; \mc{X}) = \{ f \, | \, \pd^{i}f/\pd t^{i} \in C^0(a, b\,;\mc{X}), \, 0 \leq i \leq m \},
\end{align*}
where $\pd f/\pd t$ is the time derivative in the sense of the Fr\'echet derivative in $\mc{X}$ (see e.g., \cite{Yosida-book}). 
We also define the space-time norm 
\begin{align*}
\| f \|_{L^p(a,b; \mc{X})} = 
\begin{cases}
\left( \int_a^b \| f \|_{\mc{X}}^p ds \right)^{1/p}, \quad 1 \leq p < \infty, \\
\esssup_{t \in (a,b)} \| f \|_{\mc{X}}, \quad p = \infty.
\end{cases}
\end{align*}
If a time interval $J$ is clear in context, then we use $L^p \mc{X}$ to denote $L^p(J; \mc{X})$ for simplicity. 
We define the space-time Sobolev spaces $W^{k,p}(J; \mc{X})$ for nonnegative integer $k$ and $1 \leq p \leq \infty$ as the closure of $C^k (J; \mc{X})$ with the norm $\| f \|_{W^{k,p} \mc{X}} = \sum_{i=0}^k \| \pd^i f / \pd t^i \|_{L^p \mc{X}}$. 
For simplicity we adopt the convention $\| f, g \|_\mc{X} = \| f \|_\mc{X} + \| g \|_\mc{X}$, and $\dot{u}$ is used to denote the time derivative of $f$. 

For a triangulation of $\Omega$, $\calT_h$ is used to denote a shape-regular triangulation for which $h$ is the maximum diameter of triangles (or tetrahedra) and $\mathcal{E}_h$ is the corresponding set of edges (faces), respectively. For $E \in \mathcal{E}_h$ and functions $\bs{f}, \bs{g} : \mathcal{E}_h \ra \R^n$ we define
\begin{align*}
\langle \bs{f}, \bs{g} \rangle_E = \int_E \bs{f} \cdot \bs{g} \,ds, \qquad \langle \bs{f}, \bs{g} \rangle = \sum_{E \in \mathcal{E}_h} \langle \bs{f}, \bs{g} \rangle_E. 
\end{align*}
For an integer $k \geq 0$ 
and for each $T \in \mc{T}_h$, $\mc{P}_k(T)$ is the space of polynomials of degree $\le k$ on $T$, and $\mc{P}_k(\mc{T}_h)$ denotes the space 
\algns{
\mc{P}_k(\mc{T}_h) = 
\case{ 
\LRc{q \in H^1(\Omega) \;:\; q|_T \in \mc{P}_k(T), \; T \in \mc{T}_h } \quad \text{if } k \ge 1 \\
\LRc{q \in L^2(\Omega) \;:\; q|_T \in \mc{P}_k(T), \; T \in \mc{T}_h } \quad \text{if } k = 0  \\
} .
}
For a vector space $\X$, we use $\mathcal{P}_k(G; \X)$ and $\mathcal{P}_k(\mathcal{T}_h; \X)$ to denote the space of $\X$-valued polynomials with same conditions. 

\subsection{The Biot's consolidation model}
Throughout this paper we restrict our interest on quasistatic consolidation problems and the acceleration term is ignored.
In our description of the model, $\bs{u}$ is the displacement of porous media, $p$ is the pore pressure, $\bs{f}$ is the body force, $g$ is the mass change rate of fluid. The governing equations of Biot's consolidation model with an isotropic elastic porous medium are 
\subeqns{eq:strong-eq}{
\label{eq:strong-eq1}-\div \LRp{ 2 \mu \e(\bs{u}) + (\lambda \div \ub - \alpha p) \I } &= \bs{f}, \\
\label{eq:strong-eq2} s_0 \dot{p} + \alpha \div \dot{\bs{u}} - \div (\ul{\bs{\kappa}} \nabla p) &= g, 
}
where $\mu$ and $\lambda$ are the Lam\'e coefficients, $s_0 \geq 0 $ is the constrained specific storage coefficient, $\underline{\bs{\kappa}}$ is the hydraulic conductivity tensor, $\alpha>0$ is the Biot--Willis constant which is close to 1, and $\Bbb{I}$ is the identity matrix. 
We assume that $\mu$ is uniformly bounded above and below with positive constants. We assume $\lambda$ has a uniformly positive lower bound but $\lambda$ may not have a uniform upper bound and $\lambda = +\infty$ corresponds to the incompressibility of the solid matrix. 
We assume that there are constants $c_0, c_1$ such that 
\algns{
0 \le c_0 \le s_0 (x) \le c_1 , \qquad x \in \Omega.
}
We remark that $s_0$ is related to $\alpha$, the porosity $\phi$, and the bulk moduli of the solid and fluid. 
Under the assumption that $\phi$ is uniform with $0 < \phi < \alpha$, if the solid is not incompressible, then $s_0 \ge C/\lambda$ holds with a constant $C$ of scale 1. However, $s_0$ may vanish on a subdomain if $\lambda = + \infty$ on the subdomain and the fluid is incompressible.
The hydraulic conductivity tensor $\ul{\bs{\kappa}} = \ul{\bs{\kappa}}(x)$ is positive definite with uniform lower and upper bounds $\kappa_0, \kappa_1 >0$, i.e., 
\algns{
\kappa_{0} | \xi |^2 \le \xi^T \ul{\bs{\kappa}}(x) \xi  \le \kappa_{1} | \xi |^2 , \qquad \forall \;0 \not = \xi \in \R^n,\quad  \text{a.e.} \; x \in \Omega .
}
On details of deriving these equations from physical modelling, we refer to standard porous media texts, e.g., \cite{anandarajah2010computational}. 

For well-posedness of the problem, the equations \eqref{eq:strong-eq} need appropriate boundary and initial conditions. We assume that there are partitions of $\pd \Omega$ which are 
\begin{align*}
\pd \Omega = \Gamma_p \cup \Gamma_f, \qquad \pd \Omega = \Gamma_d \cup \Gamma_t, \qquad | \Gamma_d |, |\Gamma_p| > 0
\end{align*}
where $| \Gamma |$ is the $(n-1)$-dimensional Lebesgue measure of $\Gamma$. We also assume that boundary conditions are given as
\begin{align}
\label{eq:bc}  p(t) = 0 \text{ on } \Gamma_p, \quad - \ul{\bs{\kappa}} \nabla p(t) \cdot \bs{n} = 0 \text{ on } \Gamma_f, \quad \bs{u}(t) = 0 \text{ on } \Gamma_d, \quad \underline{\bs{\sigma}}(t) \bs{n} = 0 \text{ on } \Gamma_t,
\end{align}
for all $t \in (0, T]$ where $\bs{n}$ is the outward unit normal vector field on $\pd \Omega$ and $\underline{\bs{\sigma}} := 2 \mu \e(\bs{u}) + (\lambda \div \ub - \alpha p) \I$, the Cauchy stress tensor.
Here we only consider the homogeneous boundary condition for simplicity but our method can be easily extended to problems with nonhomogeneous boundary conditions. We also assume that given initial data $p(0), \bs{u}(0)$ and $\bs{f}(0)$ satisfy the compatibility condition \eqref{eq:strong-eq1}.
Well-posedness of this system under these assumptions can be found in \cite{MR1790411}. 



\subsection{The formulation with the displacement, total and pore pressures}
In \cite{LeeEtAl2017}, a formulation of the Biot model with three unknowns was introduced in order to obtain finite element discretizations of the Biot model with parameter-robust preconditioning. 
Introduction of a new unknown $\pt := \lambda \div \bs{u} - \alpha \pf$, which will be called total pressure, gives an additional equation $\div \ub  - \lambda^{-1} (\pt + \alpha \pf)= 0$.
Therefore, we consider a system 
\subeqns{eq:upp-eq}{
\label{eq:upp-eq1} - \div \LRp{ 2 \mu \e(\bs{u}) } - \nabla \pt &= \bs{f}, \\
\label{eq:upp-eq2} \div \ub  - \lambda^{-1} (\pt + \alpha \pf)   &= 0, \\
\label{eq:upp-eq3} - \alpha \lambda^{-1} \dpt - \LRp{s_0 + \alpha^2 \lambda^{-1} } \dpf + \div ( \kapb \nabla \pf) &= -g .
}
Let us define function spaces 
\algns{ 
\Vb = \LRc{ \vb \in {H}^1 (\Omega; \R^n) \;:\; \vb|_{\Gamma_d} = 0 }, \quad 
\Qt = L^2(\Omega), \quad 
\Qf = \{ q \in H^1(\Omega) \;:\; q|_{\Gamma_p} = 0 \} ,
}
and consider the following variational form of \eqref{eq:upp-eq}: 

({\bf VP}) For initial data $(\ub(0), \pt(0), \pf(0)) \in \Vb \times \Qt \times \Qf$ satisfying 
\subeqns{eq:comp-cond}{ 
\label{eq:comp-cond-1} \LRp{2 \mu \e(\bs{u} (0)), \e(\bs{v}) } + \LRp{ \pt (0), \div \bs{v} } &= \LRp{ \bs{f} (0), \bs{v} } & & \forall  \vb \in \Vb, \\
\label{eq:comp-cond-2} \div \ub (0) - \lambda^{-1} (\pt (0) + \alpha \pf (0))   &= 0 ,
}
find 
$(\ub, \pt, \pf) \in C^1(0,T; \Vb) \times C^1(0,T; \Qt) \times C^1(0,T; \Qf)$
such that 
\subeqns{eq:weak-upp-eq}{
\label{eq:weak-upp-eq1} \LRp{2 \mu \e(\bs{u}), \e(\bs{v}) } + \LRp{ \pt , \div \bs{v} } &= \LRp{ \bs{f}, \bs{v} } & & \forall \bs{v} \in \bs{V}, \\
\label{eq:weak-upp-eq2} \LRp{ \div \ub , \qt } - \LRp{\lambda^{-1} \pt , \qt } - \LRp{ \alpha \lambda^{-1} \pf , \qt }   &= 0 & & \forall \qt \in \Qt , \\
\label{eq:weak-upp-eq3} - \LRp{ \alpha \lambda^{-1}  \dpt, \qf } - \LRp{ \LRp{s_0 + \alpha^2 \lambda^{-1} } \dpf , \qf } - \LRp{ \kapb \nabla \pf , \nabla \qf }  &= -\LRp{ g, \qf } & & \forall \qf \in \Qf .
}

\section{Energy estimates and stability}
In this section, we discuss stability of the system \eqref{eq:weak-upp-eq} with parameter-dependent norms. 
The streamline of this stability analysis will also lead to the a priori error analysis in the next section.

Let us first define parameter dependent norms 
\algns{
\nor{\vb}_{\Vb} = \LRp{2\mu \e(\ub), \e(\ub)}^{\half}, \quad \nor{\qf}_{1,\kappa} = \LRp{\kapb \nabla \qf, \nabla \qf}^{\half}, \quad \norw{\qt}{\Qt} = ((2\mu)^{-1} \qt, \qt)^{\half},
}
and for a nonnegative function (or a positive semidefinite tensor) $w$, $\nor{q }_{0,w}$ denotes $\nor{q }_{0,w} = \LRp{w q, q}^{\half}$.
We will use $\Vb'$ and $\Qt'$ to denote the dual spaces of $\Vb$ and $\Qt$, respectively. 
We also use $H^{-1}$ to denote the dual space of $H_{\Gamma_p}^1$ with the norms
\algns{
\norw{q}{-1} = \sup_{p \in H_{\Gamma_p}^1} \frac{\LRp{p, q}}{\norw{\nabla p}{0}} .
}

\begin{theorem}
Assume that 
$\fb \in W^{1,2}(0,T; \Vb')$, $g \in L^2(0,T; L^2) \cap W^{1,1}(0,T; H^{-1})$, and initial data $(\ub(0), \pt(0), \pf(0))$ satisfying \eqref{eq:comp-cond} are given. If $(\ub, \pt, \pf)$ is a solution of \eqref{eq:weak-upp-eq}, then 
\algn{
\label{eq:estm_1} &\nor{\ub}_{L^{\infty}(0,t; \Vb)} + \nor{\pt - \alpha \pf}_{L^\infty(0,t; L_{\lambda^{-1}}^2)} + \nor{\pf}_{L^\infty(0,t; L_{s_0}^2)}
+ \nor{\pf}_{L^2(0,t; H_{\kappa}^1)} \\
&\quad \lesssim \nor{\ub(0)}_{\Vb} + \nor{\pt(0) - \alpha \pf(0)}_{0,\lambda^{-1}} + \nor{\pf(0)}_{0,s_0}  +  \nor{\fb}_{W^{1,1}(0,t; \Vb')} \notag  \\ 
&\qquad + \min \LRc{{c_0}^{-\half} \nor{g}_{\LtL{1}{}{2}} , 2 \kappa_0^{-\half} \nor{g}_{\LtH{2}{}{-1}} } , \notag \\
\label{eq:estm_2} &\nor{\dot{\ub}}_{L^{2}(0,t; \Vb)} + \nor{\dot{\pt} - \alpha \dot{\pf}}_{L^2(0,t; L_{\lambda^{-1}}^2)} + \nor{\dot{\pf}}_{L^2(0,t; L_{s_0}^2)} + \nor{\pf}_{L^\infty(0,t; H_{\kappa}^1)} \\
\notag &\quad \lesssim \nor{\pf(0)}_{1,\kappa} + \nor{\dot{\fb}}_{L^2(0,t; \Vb')} \\
\notag & \qquad + \min \{ c_0^{-\half} \nor{g}_{L^2(0,t; L^2)}, \kappa_0^{-\half} \nor{g}_{W^{1,1}(0,t; H^{-1})} \} \\
\label{eq:estm_3} &\nor{\pt}_{\LtV{\infty}{\Qt}} \le C_0 ( \nor{\ub}_{L^\infty(0,t; \Vb)} + \nor{\fb}_{L^\infty(0,t; \Vb')} ) 
} 
with $C_0$ depending on $\Omega$ and $\mu$. The constants in \eqref{eq:estm_1} and \eqref{eq:estm_2} are independent of $\Omega$ and parameters.

\end{theorem}

\begin{proof}
We first prove \eqref{eq:estm_1}.
Taking $\vb = \dub$ in \eqref{eq:weak-upp-eq1}, $\qt = - \pt$ in the time differentiation of \eqref{eq:weak-upp-eq2}, $\qf = -\pf$ in \eqref{eq:weak-upp-eq3}, and adding the three equations altogether, we have 
\algn{ \label{eq:energy-eq}
\half \ddt \LRp{ \nor{\ub}_{\Vb}^2 + \nor{\pt - \alpha \pf}_{0,\lambda^{-1}}^2 + \nor{\pf}_{0,s_0}^2 } + \nor{\pf}_{1,\kappa}^2 = (\fb, \dub) + (g, \pf) .
}

Let us define $X(s) \ge 0$ and $Y(s) \ge 0$ for $s \ge 0$ as 
\algns{
X(s)^2 &= \nor{\ub(s)}_{\Vb}^2 + \nor{\pt(s) - \alpha \pf(s)}_{0,\lambda^{-1}}^2 + \nor{\pf(s)}_{0,s_0}^2 , \\
Y(s)^2 &= \int_0^s \nor{\pf(r)}_{1,\kappa}^2 dr .
}
Then integration of \eqref{eq:energy-eq} from 0 to $t$ gives
\algns{
\half(X(t)^2 - X(0)^2) + Y(t)^2  = \int_0^t \LRs{(\fb(s), \dub(s)) + (g(s),\pf(s)) } \,ds .
}
By the integration by parts in time, 
\algns{
\int_0^t \LRs{(\fb(s), \dub(s)) }\,ds = (\fb(t), \ub(t)) - (\fb(0), \ub(0)) - \int_0^t (\dot{\fb}(s), \ub(s)) \,ds ,
}
therefore we have 
\algn{ 
\notag &\half(X(t)^2 - X(0)^2) + Y(t)^2 \\
\label{eq:int-ineq1} &\quad = (\fb(t), \ub(t)) - (\fb(0), \ub(0)) + \int_0^t \LRs{(- \dot{\fb}(s), \ub(s)) + (g(s),\pf(s)) } \,ds \\
\notag &\quad \le (\nor{\fb}_{L^\infty(0,t; \Vb')} + \|{\dot{\fb}}\|_{L^1(0,t; \Vb')} ) \nor{\ub}_{L^\infty(0,t;\Vb)} \\
\notag &\qquad + \min\{ c_0^{-\half} \nor{g}_{L^1(0,t; L^2)} \nor{\pf}_{L^\infty(0,t; L_{s_0}^2)}, \kappa_0^{-\half} \nor{g}_{L^2(0,t; H^{-1})} Y(t) \} .
}

To prove \eqref{eq:estm_1} for 
\algns{
\nor{\ub}_{L^{\infty}(0,t; \Vb)} + \nor{\pt - \alpha \pf}_{L^\infty(0,t; L_{\lambda^{-1}}^2)} + \nor{\pf}_{L^\infty(0,t; L_{s_0}^2)}, 
}
note that it suffices to show the estimate for $t \in (0,T]$ such that $X(t) = \max_{s \in (0,t]} X(s)$, so we assume this maximality condition of $X(t)$. 
From the above inequality we can derive 
\algns{
&X(t)^2 + 2 Y(t)^2 \\
&\quad \le X(0)^2 + 2 \LRp{ ( \nor{\fb}_{L^\infty(0,t; \Vb')} + \|{\dot{\fb}}\|_{L^1(0,t; \Vb')}) + c_0^{-\half} \nor{g}_{L^1(0,t; L^2)} } X(t) 
}
or 
\algns{
X(t)^2 + 2 Y(t)^2 &\le X(0)^2 + 2  ( \nor{\fb}_{L^\infty(0,t; \Vb')} + \|{\dot{\fb}}\|_{L^1(0,t; \Vb')} ) X(t) \\
& \quad + 2 \kappa_0^{-\half} \nor{g}_{L^2(0,t; H^{-1})} Y(t) .
}
Applying Young's inequality, we can obtain either
\algns{
X(t)^2 
&\le 2 X(0)^2 + 4 \LRp{ ( \nor{\fb}_{L^\infty(0,t; \Vb')} + \|{\dot{\fb}}\|_{L^1(0,t; \Vb')}) + {c}_0^{-\half} \nor{g}_{L^1(0,t; L^2)} }^2 
}
or 
\algns{
X(t)^2 
\le 2 X(0)^2 + 4 ( \nor{\fb}_{L^\infty(0,t; \Vb')} + \|{\dot{\fb}}\|_{L^1(0,t; \Vb')} )^2 + 2 \kappa_0^{-1} \nor{g}_{\LtH{2}{}{-1}}^2 ,
}
thus 
\algn{ \label{eq:Xt-estm}
X(t) 
&\lesssim X(0) + ( \nor{\fb}_{L^\infty(0,t; \Vb')} + \|{\dot{\fb}}\|_{L^1(0,t; \Vb')} ) \\
\notag &\quad + \min \{ {c}_0^{-\half} \nor{g}_{L^1(0,t; L^2)},  \kappa_0^{-\half} \nor{{g}}_{L^2(0,t; H^{-1})} \} .
}
Note that $\nor{\ub}_{L^{\infty}(0,t; \Vb)}, \nor{\pt - \alpha \pf}_{L^\infty(0,t; L_{\lambda^{-1}}^2)} , \nor{\pf}_{L^\infty(0,t; L_{s_0}^2)} \le X(t)$ due to the maximality of $X(t)$. Then \eqref{eq:estm_1} for 
\algns{
\nor{\ub}_{L^{\infty}(0,t; \Vb)} + \nor{\pt - \alpha \pf}_{L^\infty(0,t; L_{\lambda^{-1}}^2)} + \nor{\pf}_{L^\infty(0,t; L_{s_0}^2)}
}
follows from the above inequality. We remark that this estimate can be extended to all $t \in (0, T]$, and we will use this estimate for general $t$ below.

To complete the proof of \eqref{eq:estm_1}, we need to estimate $Y(t)$ without the assumption $X(t) = \max_{s \in (0,t]} X(s)$. From \eqref{eq:int-ineq1} we get
\algns{
Y(t)^2 \le \half X(0)^2 + \LRp{ ( \nor{\fb}_{L^\infty(0,t; \Vb')} + \|{\dot{\fb}}\|_{L^1(0,t; \Vb')}) + c_0^{-\half} \nor{g}_{\LtL{1}{}{2}} } X(\bar{t})
}
or
\algns{
\half Y(t)^2 \le \half X(0)^2 + (\nor{\fb}_{L^\infty(0,t; \Vb')} + \|{\dot{\fb}}\|_{L^1(0,t; \Vb')}) X(\bar{t}) + \half \kappa_0^{-\half} \nor{g}_{\LtH{2}{}{-1}}
}
where $X(\bar{t}) = \max_{s \in [0,t]} X(s)$. Combining these with \eqref{eq:Xt-estm}, the proof of \eqref{eq:estm_1} is completed.


We now prove \eqref{eq:estm_2}. 
For this, we take $\qf = - \dpf$ in \eqref{eq:weak-upp-eq3}, $\vb = \dub$ in the time derivative of \eqref{eq:weak-upp-eq1}, $\qt = - \dpt$ in the time derivative of \eqref{eq:weak-upp-eq2}, and add all the equations together. Then we have 
\algn{ \label{eq:d-energy-eq}
&\nor{\dub(t)}_{\Vb}^2 + \nor{\dpt(t) - \alpha \dpf(t)}_{0,\lambda^{-1}}^2 + \nor{\dpf(t)}_{0,s_0}^2 + \half \ddt \nor{\pf(t)}_{1,\kappa}^2 \\
\notag &\quad = (\dot{\fb}(t), \dub(t)) + (g(t), \dpf(t)) .
}
If $s_0$ is non-degenerate with $s_0 \ge c_0 > 0$, by Young's inequality, 
\algns{
&\half \nor{\dub(t)}_{\Vb}^2 + \nor{\dpt(t) - \alpha \dpf(t)}_{0,\lambda^{-1}}^2 + \half \nor{\dpf(t)}_{0,s_0}^2 + \half \ddt \nor{\pf(t)}_{1,\kappa}^2 \\
&\quad \le \half \nor{\dot{\fb}(t)}_{\Vb'}^2 + \half c_0^{-1} \nor{g(t)}_0^2 .
}
Integrating this from 0 to $t$ gives
\mltln{ \label{eq:stab-energy-estm1}
\nor{\pf(t)}_{1,\kappa}^2 + \int_0^t \LRs{\nor{\dub(s)}_{\Vb}^2 + 2\nor{\dpt(s) - \alpha \dpf(s)}_{0,\lambda^{-1}}^2 + \nor{\dpf(s)}_{0,s_0}^2 } \,ds \\
\quad \le \nor{\pf(0)}_{1,\kappa}^2 + \int_0^t \LRs{ \nor{\dot{\fb}(s)}_{\Vb'}^2 + c_0^{-1} \nor{g(s)}_0^2 } \,ds .
}
When $s_0$ is degenerate, we integrate \eqref{eq:d-energy-eq} from 0 to $t$ and get 
\algn{
\label{eq:aux_estm} &\int_0^t \LRs{\nor{\dub(t)}_{\Vb}^2 + \nor{\dpt(t) - \alpha \dpf(t)}_{0,\lambda^{-1}}^2 + \nor{\dpf(t)}_{0,s_0}^2 } ds + \half \nor{\pf(t)}_{1,\kappa}^2 \\
\notag &\quad = \half \nor{\pf(0)}_{1,\kappa}^2 + \int_0^t \LRs{(\dot{\fb}(s), \dub(s)) + (g(s), \dpf(s))} \,ds \\
\notag &\quad = \half \nor{\pf(0)}_{1,\kappa}^2 + \int_0^t \LRs{(\dot{\fb}(s), \dub(s)) - (\dot{g}(s), \pf(s))} \,ds + (g(t), \pf(t)) - (g(0), \pf(0)) .
}
Since $\nor{\pf(t)}_{1,\kappa} \le \nor{\pf}_{L^\infty(0,t; H_{\kappa}^1)}$, 
without loss of generality, we may assume that $\nor{\pf(t)}_{1,\kappa} = \nor{\pf}_{L^\infty(0,t; H_{\kappa}^1)}$. Then the above formula gives 
\algns{
&\int_0^t \LRs{\nor{\dub(t)}_{\Vb}^2 + \nor{\dpt(t) - \alpha \dpf(t)}_{0,\lambda^{-1}}^2 + \nor{\dpf(t)}_{0,s_0}^2 } ds + \half \nor{\pf(t)}_{1,\kappa}^2 \\
&\quad \le \half \nor{\pf(0)}_{1,\kappa}^2 + \nor{\dot{\fb}}_{L^2(0,t; \Vb')} \nor{\dub}_{L^2(0,t; \Vb)} \\
&\qquad + \kappa_0^{-\half} \LRp{\nor{\dot{g}}_{L^1(0,t; H^{-1})} + \nor{g}_{L^\infty(0,t;H^{-1})} } \nor{\pf}_{L^\infty(0,t; H_{\kappa}^1)} .
}
By Young's inequality, we have 
\mltln{ \label{eq:stab-energy-estm2}
\int_0^t \LRs{\nor{\dub(t)}_{\Vb}^2 + 2\nor{\dpt(s) - \alpha \dpf(s)}_{0,\lambda^{-1}}^2 + 2\nor{\dpf(s)}_{0,s_0}^2 } ds + \half \nor{\pf(t)}_{1,\kappa}^2 \\
\quad \le  \nor{\pf(0)}_{1,\kappa}^2 + \nor{\dot{\fb}}_{L^2(0,t; \Vb')}^2 + 2\kappa_0^{-1} \LRp{\nor{\dot{g}}_{L^1(0,t; H^{-1})} + \nor{g}_{L^\infty(0,t;H^{-1})} }^2 .
}
Combining \eqref{eq:stab-energy-estm1} and \eqref{eq:stab-energy-estm2}, we have 
\algns{
\nor{\pf(t)}_{1,\kappa}  
&\lesssim \nor{\pf(0)}_{1,\kappa} + \nor{\dot{\fb}}_{L^2(0,t; \Vb')} \\
&\quad + \min \{ c_0^{-\half} \nor{g}_{L^2(0,t; L^2)}, \kappa_0^{-\half} (\nor{\dot{g}}_{\LtH{1}{}{-1}} + \nor{g}_{\LtH{\infty}{}{-1}}) \} ,
}
so \eqref{eq:estm_2} for $\nor{\pf}_{L^\infty(0,t; H_{\kappa}^1)}$ is proved. 
We can also estimate 
$$\int_0^t \left[{\nor{\dub(s)}_{\Vb}^2 + \nor{\dpt(s) - \alpha \dpf(s)}_{0,\lambda^{-1}}^2 + \nor{\dpf(s)}_{0,s_0}^2 } \right] ds$$ 
from \eqref{eq:stab-energy-estm1} and \eqref{eq:stab-energy-estm2} with the estimate of $\nor{\pf}_{L^\infty(0,t; H_{\kappa}^1)}$. The argument is completely analogous to the estimate of $\nor{\pf}_{L^2(0,t; H_{\kappa}^1)}$, so we omit details.

Finally, we prove \eqref{eq:estm_3}. From the inf-sup condition 
\algns{
\inf_{0 \not = \qt \in \Qt} \sup_{0 \not = \vb \in \Vb} \frac{(\div \vb, \qt)}{\nor{\e(\vb)}_{0} \nor{\qt}_{0}} \ge C,
}
for any given $\qt$, there exists $\vb \in \Vb$ such that $(\div \vb, \qt') = (\qt, \qt')$ for all $\qt' \in \Qt$,  and $\nor{\e(\vb)}_0 \le C_{\Omega} \nor{\qt}_0$ with $C_{\Omega}$ depending only on $\Omega$. 
If we take $\vb$ as such an element in $\Vb$ with $\qt = \frac{1}{2\mu} \pt$, then we can check that 
\algns{
\nor{\vb}_{\Vb} \le \sqrt{2 \mu_1} \nor{\vb}_1 \le C_{\Omega} \sqrt{\frac{\mu_1}{\mu_0}} \nor{\pt}_{\Qt}
}
with $\mu_1 := \nor{\mu}_{L^\infty}$, $\frac{1}{\mu_0} := \nor{\frac{1}{\mu}}_{L^\infty}$. 
From the estimate of $\nor{\ub}_{L^{\infty}(0,t; \Vb)}$ and \eqref{eq:weak-upp-eq1}, 
\algns{
\nor{\pt}_{\Qt}^2 = - (2\mu \e(\ub), \e(\vb)) + (\fb, \vb) \le C_{\Omega} \sqrt{ \frac{\mu_1}{\mu_0}} \LRp{ \nor{\ub}_{\Vb} + \nor{\fb}_{\Vb'} } \nor{\pt}_{\Qt} 
}
holds, and therefore  
\algn{ \label{eq:stab-pt-estm}
\nor{\pt}_{\LtV{\infty}{\Qt}} \le C_{\Omega} \sqrt{ \frac{\mu_1}{\mu_0}} ( \nor{\ub}_{L^\infty(0,t; \Vb)} + \nor{\fb}_{L^\infty(0,t; \Vb')} ) .
}
\end{proof}

\section{Discretization with finite elements}
In this section we discuss finite element discretization of \eqref{eq:weak-upp-eq} and the a priori error analysis of numerical solutions. 
We are interested in discretizations which are robust for the parameters including arbitrarily large $\lambda>0$, and only nonnegative $s_0 \ge 0$. 
Note that the limit case $\lambda = \infty$ decouples \eqref{eq:weak-upp-eq} into two separate problems, the Stokes equation and a time-dependent Darcy flow problems. 
Therefore, it is natural to combine two finite element methods, one for the Stokes equation for $(\ub, \pt)$ and the other for the Darcy flow problems for $\pf$. 

For discretizations of the Stokes equation, standard mixed methods with conforming finite elements are natural choices but stabilized methods for the Stokes equation are sometimes preferred due to their smaller number of degrees of freedom. Therefore we propose formulations covering some low order stabilized methods for discretization of $(\ub, \pt)$ with the a priori error analysis. The parameter $\mu$ is assumed to be 1 in the model problem of Stokes equations. However, $\mu$ is a function in $\Omega$ with large parameter value in most practical poroelasticity problems, so we assume that $1 \lesssim \mu_{\min} \le \mu \le \mu_{\max}$ and $\mu_{\max}/ \mu_{\min}$ is bounded above and below in $\Omega$. 

For discretizations of $\pf$, the standard method with Lagrange finite elements is the simplest numerical method but it does not give numerical solutions with local mass conservation. In this paper we use the enriched Galerkin method that we can obtain a locally mass conservative flux via local post-processing. However, our error analysis can be extended to any discretization methods of the Poisson equation including continuous and various discontinuous Galerkin methods.


\subsection{Finite element methods for the Stokes and Poisson equations}
In this subsection we introduce the mixed and stabilized methods for the Stokes equation of $(\ub, \pt)$ and the Lagrange finite elements for the Poisson equation of $\pf$. In this section we denote $\Vbh$, $\Qth$, $\Qfh$ the finite element spaces for the unknowns $\ub$, $\pt$, $\pf$, and assume that $\Vbh \subset \Vb$, $\Qth \subset \Qt$, and $\Qfh \subset \Qf$. We will use $k_{\ub}$, $k_{\pt}$, $k_{\pf}$ to denote the maximum polynomial approximation orders of $\Vbh$, $\Qth$, $\Qfh$ with the $L^2$ norm.

To describe the mixed and stabilized methods of $\Vbh$ and $\Qth$, let us consider an auxiliary problem to find $(\ub, p) \in \Vb \times \Qt$ such that 
\algn{ \label{eq:aux-stokes}
\LRp{2 \mu \e(\ub), \e(\vb) } + (\pt, \div \vb) = (\fb_1, \vb), \qquad 
(\div \ub, \qt) = (f_2, \qt) 
}
for all $(\vb, \qt) \in \Vb \times \Qt$. First, we can use stable mixed finite elements $(\Vbh, \Qth)$, i.e., the pair $(\Vbh, \Qth)$ satisfies the inf-sup condition
\algn{ \label{eq:mixed-inf-sup}
\inf_{0 \not = \qt \in \Qth} \sup_{0 \not = \vb \in \Vbh} \frac{(\div \vb, \qt)}{\nor{\nabla\vb}_{0} \nor{\qt}_{0}} \ge C >0
}
with a constant $C$ independent of $h$.
A similar inf-sup condition holds with denominator $\norw{\e(\vb)}{\Vb}\norw{\qt}{\Qt}$ by rescaling of norms, and the inf-sup constant depends on the constant of Korn's inequality and $\mu_{\max} / \mu_{\min}$.
For stabilized methods for \eqref{eq:aux-stokes}, we consider the stabilized methods of the form
\algns{
\mc{B}(\ubh, \pth; \vb, \qt) &:= (2 \mu \e(\ubh), \e(\vb)) + (\pth, \div \vb) + (\div \ubh, \qt) - s_h(\pth, \qt), \\ 
F(\vb, \qt) &:= (\fb_1, \vb) + (f_2, \qt) + \tilde{s}_h(\fb_1, \qt)
}
with some bilinear and linear forms $s_h$ and $\tilde{s}_h$ on $\Vbh \times \Qth$ such that 
\algns{
| s_h(\pt, \qt) | \lesssim \nor{\pt}_{\Qt} \nor{\qt}_{\Qt}. 
}
The discretization of \eqref{eq:aux-stokes} is to find $(\ubh, \pth) \in \Vbh \times \Qth$ such that 
\algn{ \label{eq:stab-stokes}
\mc{B}(\ubh, \pth; \vb, \qt) = F(\vb, \qt) \qquad (\vb, \qt) \in \Vbh \times \Qth .
}
We assume that this discretization is consistent (with sufficiently regular exact solutions) and also assume 
that an inf-sup condition 
\algn{ \label{eq:stab-inf-sup}
\inf_{(\ub, \pt) \in \Vbh \times \Qth} \sup_{(\vb, \qt) \in \Vbh \times \Qth} \frac{\mc{B}(\ub, \pt; \vb, \qt)}{ (\norw{\ub}{\Vb} + \norw{\pt}{\Qt}) (\norw{\vb}{\Vb} + \norw{\qt}{\Qt}) } \ge C > 0 
}
holds with $C$ independent of $h$ and parameters. 

We here remark that there are known stabilized methods satisfying \eqref{eq:stab-inf-sup}, for example, 
\begin{subequations}
\label{eq:stab-method1}
\algn{
\Vbh &= \mc{P}_1(\mc{T}_h; \R^n), & \Qth &= \mc{P}_0(\mc{T}_h), \\
\quad s_h (\pt, \qt) &= \frac{\gamma_2}{2\mu} \sum_{e \in \mc{E}_h} h_e^{-1} \LRa{\jump{\pt}, \jump{\qt}}_e , & \tilde{s}_h &= 0, 
}
\end{subequations}
with $\jump{\qt}$, the jump of $\qt$ on edges/faces (cf. \cite{KechkarSilvester1992}), and 
\begin{subequations}
\label{eq:stab-method2}
\algn{
\Vbh &= \mc{P}_1(\mc{T}_h ; \R^n), & \Qth &= \mc{P}_1(\mc{T}_h), \\
s_h(\pt, \qt) &= \frac{\gamma_2}{2\mu} \sum_{T \in \mc{T}_h} h_T^2 \LRp{\nabla \pt, \nabla \qt }_{T}, &
\tilde{s}_h(\fb, \qt) &= - \frac{\gamma_2}{2\mu} \sum_{T \in \mc{T}_h} h_T^2 (\fb, \nabla \qt)
}
\end{subequations}
where $\gamma_2 >0$ is a parameter depending on the shape regularity of meshes. These stabilization methods were proposed in \cite{BrezziPitkaranta1984} and \cite{KechkarSilvester1992}, respectively. 
For more on stabilized methods for the Stokes equation, we refer to \cite{FrancaHughesStenberg2008}.

For $\Qfh$ we use the standard Lagrange finite elements.

\subsection{Semidiscrete error analysis}
The semidiscrete formulation of \eqref{eq:weak-upp-eq} is to find 
$(\ubh, \pth, \pfh) \in C^1(0,T;\Vbh ) \times C^1(0,T; \Qth ) \times C^1(0,T; \Qfh )$
such that 
\subeqns{eq:weak-upp-semi}{
\label{eq:weak-upp-semi1} \LRp{2 \mu \e(\ubh ), \e(\vb ) } + \LRp{ \pth , \div \vb } &= (\fb, \vb), \\ 
\label{eq:weak-upp-semi2} \LRp{ \div \ubh , \qt } - s_h \LRp{\pth, \qt} - \LRp{ \lambda^{-1} \pth , \qt } - \LRp{\alpha \lambda^{-1} \pfh , \qt}   &= \tilde{s}_h(\fb, \qt), \\
\label{eq:weak-upp-semi3} - \LRp{ \alpha \lambda^{-1} \dpth , \qf } - \LRp{\LRp{s_0 + \alpha^2 \lambda^{-1} } \dpfh , \qf } - \LRp{\kapb \nabla \pfh , \nabla \qf }  &= \LRp{ g, \qf } 
}
for any $\vb \in \Vbh$, $\qt \in \Qth$, $\qf \in \Qfh$. 
It is obvious that $s_h = \tilde{s}_h=0$ if we use mixed methods for $(\ubh, \pt)$. 

Suppose that $(\ub, \pt, \pf)$ is an exact solution of \eqref{eq:weak-upp-eq} and $(\ubh, \pth, \pfh)$ is a numerical solution of \eqref{eq:weak-upp-semi}, and 
define 
\algns{
\eu(t) := \ub (t) - \ubh(t), \quad \ept(t) := \pt(t) - \pth(t), \quad \epf(t) := \pf(t) - \pfh(t) .
}
For some interpolations $(\Pi_h^{\Vb} \ub (t), \Pi_h^{\Qt} \pt(t), \Pi_h^{\Qf} \pf(t)) \in \Vbh \times \Qth \times \Qfh$, which will be defined below, we split the errors into two parts as
\algn{
\label{eq:u-split} 
\eu(t) &= \eui(t) + \euh(t) := (\bs{u}(t) - \Pi_h^{\Vb} \bs{u}(t)) + (\Pi_h^{\Vb} \bs{u}(t) - \bs{u}_h (t)), \\
\label{eq:pt-split} 
\ept(t) &= \epti(t) + \epth(t) := (\pt(t) - \Pi_h^{\Qt} \pt(t)) + (\Pi_h^{\Qt} \pt(t) - \pth(t) ), \\
\label{eq:pf-split} 
\epf(t) &= \epfi(t) + \epfh(t) := (\pf(t) - \Pi_h^{\Qf} \pf(t)) + (\Pi_h^{\Qf} \pf(t) - \pfh(t) ).
}
We define $\Pi_h^{\Vb} \ub (t)$ and $\Pi_h^{\Qt} \pt(t)$ as the solution of auxiliary problem: \\
{\bf (AP1)} Find $(\Pi_h^{\Vb} \ub (t), \Pi_h^{\Qt} \pt(t)) \in \Vbh \times \Qth$ such that 
\algns{ 
\LRp{2 \mu \e(\Pi_h^{\Vb} \ub (t)), \e(\vb ) } + \LRp{ \Pi_h^{\Qt} \pt(t) , \div \vb } &= (\fb(t), \vb), \\ 
\LRp{ \div \Pi_h^{\Vb} \ub (t), \qt } - s_h \LRp{ \Pi_h^{\Qt} \pt(t), \qt} &= (\div \ub(t), \qt) + \tilde{s}_h(\fb(t), \qt)
}
for any $(\vb, \qt) \in \Vbh \times \Qth$. \\
The stability of mixed methods (when $s_h = \tilde{s}_h = 0$) or stabilized methods guarantees the well-posedness of this problem, and furthermore, standard error analyses of mixed or stabilized methods for the Stokes equation give
\algn{ \label{eq:up-intp}
\norw{\ub(t) - \Pi_h^{\Vb} \ub (t)}{\Vb} + \|\pt(t) - \Pi_h^{\Qt} \pt(t)\|_{\Qt} \lesssim h^m (\norw{\ub(t)}{m+1} + \norw{\pt(t)}{m}) 
}
with $m \le \max \{ k_{\ub}-1, k_{\pt} \}$ which depends on the regularities of $\ub(t)$ and $\pt(t)$.

We define $\Pi_h^{\Qf} \pf(t)$ as the solution of another auxiliary problem: \\
{\bf (AP2)} Find $\Pi_h^{\Qf}\pf(t) \in \Qfh$ such that 
\algns{
(\kapb \nabla \Pi_h^{\Qf}\pf, \nabla \qf) = (\kapb \nabla \pf, \nabla \qf) \qquad \forall \qf \in \Qfh .
}
It is well-known that 
\algn{ \label{eq:pf-H1}
\| \pf(t) - \Pi_h^{\Qf} \pf(t) \|_{1, \kappa} \lesssim \kappa_0^{-\half} h^m \norw{\pf(t)}{m+1} 
}
holds with $m \le k_{\pf} - 1$ depending on the regularity of $\pf(t)$. 
If $\Omega$ satisfies the full elliptic regularity assumption and $\ul{\bs{\kappa}}$ is a Lipschitz continuous scalar field on $\Omega$, then 
\algn{ \label{eq:pf-L2}
\|{\pf(t) - \Pi_h^{\Qf} \pf(t)}\|_0 \lesssim h \norw{\pf(t) - \Pi_h^{\Qf} \pf(t)}{1,\kappa} 
}
holds as well.

Before we prove the a priori error analysis we discuss 
compatible numerical initial data. Note that \eqref{eq:weak-upp-semi1}, \eqref{eq:weak-upp-semi2} are algebraic equations, so our problem is a system of differential algebraic equations. 
When the backward Euler method is used for time discretization, compatible numerical data is not significant because the algebraic equation will be satisfied after one time step. 
However, numerical initial data satisfying this algebraic equation can be important for stability of numerical methods when other time discretization methods such as the Crank--Nicolson method are used. In order to have compatible numerical initial data, we can use the solution of 
\algns{
\LRp{2 \mu \e(\ubh ), \e(\vb ) } + \LRp{ \pth , \div \vb } &= (\fb(0), \vb), \\ 
\LRp{ \div \ubh , \qt } - s_h \LRp{\pth, \qt} - \LRp{ \lambda^{-1} \pth , \qt } - \LRp{\alpha \lambda^{-1} \pfh , \qt}   &= \tilde{s}_h(\fb(0), \qt), \\
- \LRp{\alpha^2 \lambda^{-1} \pth , \qf } - \LRp{\kapb \nabla \pfh , \nabla \qf }  &= - \LRp{\alpha^2 \lambda^{-1} \pt(0) , \qf } \\
&\quad - \LRp{\kapb \nabla \pf(0) , \nabla \qf } 
}
as numerical initial data. Since this is a stabilized saddle point problem with inf-sup condition, it is rather standard to show that the numerical initial data from this problem is a good approximation of initial data of the continuous problem.

In the theorem below we assume that the exact solutions are sufficiently regular and maximum approximation orders can be achieved in the Bramble--Hilbert lemma for simplicity of presentation.
In addition, we also assume that $\Pi^{\Qf} \pf$ is an approximation of $\pf$ with optimal order in the $L^2$ norm, i.e., \eqref{eq:pf-L2} holds.


\begin{theorem} \label{thm:eh-estm}
Suppose that $(\ub, \pt, \pf)$ is the solution of \eqref{eq:weak-upp-eq} with initial data $(\ub(0), \pt(0), \pf(0))$, and $(\ubh, \pth, \pfh)$ is the solution of \eqref{eq:weak-upp-semi} with numerical initial data $(\ubh(0), \pth(0), \pfh(0)) \in \Vbh \times \Qth \times \Qfh$ satisfying \eqref{eq:weak-upp-semi1}, \eqref{eq:weak-upp-semi2}, and 
\algn{
\label{eq:init-approx1} 
\norw{\pt(0) - \pth(0)}{\Qt} \lesssim h^{k_{\pt}} \norw{\pt(0)}{k_{\pt}} , \\
\label{eq:init-approx2} 
\norw{\pf(0) - \pfh(0)}{0} \lesssim h^{k_{\pf}} \norw{\pf(0)}{k_{\pf}} .
}
Then
\algn{
\label{eq:semi-error1} 
&\nor{\Pi_h^{\Vb} \ub - \ubh}_{L^\infty(0,t; \Vb)} + \norw{\Pi_h^{\Qt} \pt - \pth}{L^\infty(0,t; \Qt)} \\
\notag &+ \nor{\Pi_h^{\Qf}\pf - \pfh}_{L^\infty(0,t; L_{s_0}^2)} + \nor{\Pi_h^{\Qf}\pf - \pfh}_{L^2(0,t; H_{\kappa}^1)} \\
\notag & \quad \lesssim h^k \LRp{\nor{\pt(0) }_{H^k} + \nor{\pf(0)}_{H^k} + \nor{\dpt}_{L^1(0,t; H^k)} + \nor{\dpf}_{L^1(0,t; H^k)} }
}
and
\algn{
\label{eq:semi-error2} 
&\nor{\Pi_h^{\Vb} \dub - \dubh}_{L^2(0,t; \Vb)} + \norw{\Pi_h^{\Qt}\dpt - \dpth}{L^2(0,t; \Qt)}  \\
\notag &+ \nor{\Pi_h^{\Qf}\dpf - \dpfh}_{L^2(0,t; L_{s_0}^2)} + \nor{\Pi_h^{\Qf} \pf - \pfh}_{L^\infty(0,t; H_{\kappa}^1)} \\
\notag & \quad \lesssim \nor{\Pi_h^{\Qf}\pf(0) - \pfh(0)}_{1, \kappa} + h^{k} \norw{\dpt, \dpf}{\LtH{2}{}{k}} 
}
hold with $k = \min \{ k_{\pt}, k_{\pf} \}$.
\end{theorem}
\begin{proof}
The difference of \eqref{eq:weak-upp-eq} and \eqref{eq:weak-upp-semi} gives 
\algns{
\LRp{2 \mu \e(\eu), \e(\vb ) } + \LRp{ \ept , \div \vb } &= 0, \\ 
\LRp{ \div \eu , \qt } + s_h \LRp{\pth, \qt} - \LRp{ \lambda^{-1} \ept , \qt } - \LRp{\alpha \lambda^{-1} \epf , \qt}   &= - \tilde{s}_h(\fb, \qt), \\
- \LRp{\alpha \lambda^{-1} \dot{e}_{\pt} , \qf } - \LRp{\LRp{s_0 + \alpha^2 \lambda^{-1} } \dot{e}_{\pf}, \qf } - \LRp{ \kapb \nabla \epf , \nabla \qf }  &= 0 .
}
From the decomposition \eqref{eq:u-split}--\eqref{eq:pf-split} and the equations of {\bf (AP1)}, {\bf (AP2)}, we have reduced error equations 
\subeqns{eq:err-eq}{
\label{eq:err-eq1} 
& \LRp{2 \mu \e(\euh), \e(\vb ) } + \LRp{ \epth , \div \vb } = 0 , \\
\label{eq:err-eq2} 
&\LRp{ \div \euh , \qt } - s_h \LRp{\epth, \qt} - \LRp{ \lambda^{-1} (\epth - \alpha \epfh), \qt }   \\
\notag & \qquad = \LRp{ \lambda^{-1} \epti , \qt } + (\alpha \lambda^{-1} \epfi , \qt) , \\
\label{eq:err-eq3} 
&- \LRp{ \alpha \lambda^{-1} \depth , \qf } - \LRp{\LRp{s_0 + \alpha^2 \lambda^{-1} } \depfh , \qf } - \LRp{ \kapb \nabla \epfh , \nabla \qf }  \\
\notag &\qquad = \LRp{ \alpha \lambda^{-1} \depti , \qf } - \LRp{\LRp{s_0 + \alpha^2 \lambda^{-1} } \depfi , \qf }  
}
for any $\vb \in \Vbh$, $\qt \in \Qth$, $\qf \in \Qfh$.

{\bf Proof of \eqref{eq:semi-error1}  }: We take $\vb = \deuh$ in \eqref{eq:err-eq1}, $\qt = -\epth$ in the time derivative of \eqref{eq:err-eq2}, $\qf = - \epfh$ in \eqref{eq:err-eq3}, and add them altogether. Then we have 
\algn{ \label{eq:err-energy-eq}
\half \ddt \LRp{ \nor{\euh}_{\Vb}^2 + s_h(\epth, \epth) + \nor{\epth - \alpha \epfh}_{0,\lambda^{-1}}^2 + \nor{\epfh}_{0,s_0}^2 } + \| \epfh \|_{1, \kappa}^2 \\
\notag =  -\LRp{ \lambda^{-1} (\depti - \alpha \depfi) , \epth - \alpha \epfh} 
+ \LRp{s_0 \depfi , \epfh } .
}
Defining 
\algns{
X(s)^2 &= \nor{\euh(s)}_{\Vb}^2 + s_h(\epth(s), \epth(s)) + \nor{\epth(s) - \alpha \epfh(s)}_{0,\lambda^{-1}}^2 + \nor{\epfh(s)}_{0,s_0}^2, 
}
and integrating \eqref{eq:err-energy-eq} from 0 to $t$, we have 
\algn{
\label{eq:err-int-ineq1} &\half (X(t)^2 - X(0)^2) + \int_0^t \| \epfh(s) \|_{1,\kappa}^2 \,ds,  \\
\notag & \quad = \int_0^t \LRs{-\LRp{ \lambda^{-1} (\depti(s) - \alpha \depfi(s)) , \epth(s) - \alpha \epfh(s)} + \LRp{s_0 \depfi (s), \epfh (s) } } ds \\
\notag & \quad \le \norw{ \depti - \alpha \depfi}{\LtL{1}{\lambda^{-1}}{2}} \norw{\epth - \alpha \epfh}{\LtL{\infty}{\lambda^{-1}}{2}} \\
\notag & \qquad + \norw{ \depfi}{\LtL{1}{s_0}{2}} \norw{\epfh}{\LtL{\infty}{s_0}{2}} .
}
Adopting the argument of the estimate of $X(t)$ in the previous section, we may assume that $X(t) = \max_{s \in (0,t]} X(s)$ without loss of generality. Then 
\algns{
\half X(t)^2 
\le \half X(0)^2 + \max \left\{ \norw{ \depti - \alpha \depfi}{\LtL{1}{\lambda^{-1}}{2}}, \norw{ \depfi}{\LtL{1}{s_0}{2}} \right\} X(t) . 
}
By Young's inequality and the arithmetic-geometric mean inequality, we can obtain
\algns{
X(t) \le X(0) + 2 \max \left\{ \norw{ \depti - \alpha \depfi}{\LtL{1}{\lambda^{-1}}{2}}, \norw{ \depfi}{\LtL{1}{s_0}{2}} \right\} .
}
As a corollary, assuming the exact solution is sufficiently smooth, we obtain 
\mltln{ \label{eq:euh-estm}
\norw{\euh}{L^\infty(0, t; \Vb) } + \max_{s \in [0,t]} s_h(\epth, \epth)^{\half} + \nor{\epth - \alpha \epfh}_{L^\infty(0,t; L_{\lambda^{-1}}^2)} \\
+ \norw{\epfh}{\LtL{\infty}{s_0}{2}} \lesssim X(0) + h^{k} \norw{ \dpt, \dpf }{\LtH{1}{}{k}} 
}
where $k = \min \{ k_{\pt}, k_{\pf} \}$. 
Note that the implicit constant in this estimate is independent of parameter scales, i.e.,
for large $\mu$, arbitrarily large $\lambda$, small $\kappa_{0}$ and $\kappa_1$, and small or degenerate $s_0$. 
For mixed methods, the equation \eqref{eq:err-eq1} and the inf-sup condition \eqref{eq:mixed-inf-sup} can be used to obtain
\algn{ \label{eq:pt-estm}
\norw{\epth}{L^\infty(0, t; \Qt)} \lesssim X(0) + h^{k} \norw{ \dpt, \dpf }{\LtH{1}{}{k}} ,\qquad k = \min \{ k_{\pt}, k_{\pf} \} .
}
In case of stabilized methods, for any $t \in (0, T]$, there exists $(\vb, \qt)$ such that $\nor{\vb}_{\Vb} + \nor{\qt}_{\Qt} \le 1$ and 
\mltlns{ 
\nor{\euh(t)}_{\Vb} + \nor{\epth(t)}_{\Qt} \lesssim \\
\LRp{2 \mu \e(\euh(t)), \e(\vb ) } + \LRp{ \epth (t), \div \vb } + \LRp{ \div \euh (t), \qt } - s_h \LRp{\epth (t), \qt} .
}
Using this $(\vb, \qt)$ with \eqref{eq:err-eq1} and \eqref{eq:err-eq2}, we get 
\algn{
\label{eq:upt-estm} &\nor{\euh(t)}_{\Vb} + \nor{\epth(t)}_{\Qt} \\
\notag &\quad \lesssim \LRp{ \lambda^{-1} (\epth (t) - \alpha \epfh(t)), \qt } + \LRp{ \lambda^{-1} \epti (t), \qt } + (\alpha \lambda^{-1} \epfi (t), \qt) , \\
\notag &\quad \lesssim \nor{\epth (t) - \alpha \epfh (t) }_{0,\lambda^{-1}} + \nor{\epti (t) - \alpha \epfi(t)}_{0,\lambda^{-1}} \\
\notag &\quad \lesssim X(0) + h^{k} \norw{ \dpt, \dpf }{\LtH{1}{}{k}} , \qquad k = \min \{ k_{\pt}, k_{\pf} \}, 
}
where we used \eqref{eq:euh-estm} in the last inequality.

To estimate $\nor{\epfh}_{L^2(0,t; H_{\kappa}^1)}$, we use \eqref{eq:err-int-ineq1} and get 
\algns{
\half X(t)^2 + \int_0^t \nor{\epfh(s)}_{1,\kappa}^2 \,ds &\le \half X(0)^2 + \norw{ \depti - \alpha \depfi}{\LtL{1}{\lambda^{-1}}{2}} X(t) \\
&\quad + \norw{ \depfi}{\LtL{1}{s_0}{2}} X(t) .
}
By Young's inequality, 
\algn{ \label{eq:epfh-H1-estm}
\nor{\epfh}_{L^2(0,t; H_{\kappa}^1)} \lesssim X(0) + h^{k} \norw{ \dpt, \dpf }{\LtH{1}{}{k}} , \qquad k = \min \{ k_{\pt}, k_{\pf} \} .
}

To complete the proof, we need to estimate $X(0)$. Recall that $(\ubh(0), \pth(0), \pfh(0))$ satisfies \eqref{eq:weak-upp-semi1} and \eqref{eq:weak-upp-semi2} at $t=0$.
Recall also that $(\Pi_h^{\Vb} \ub(0), \Pi_h^{\Qt} \pt(0))$ satisfies {\bf (AP1)} at $t= 0$. Noting that $\div \ub(0) = \lambda^{-1} \pt(0) + \lambda^{-1} \alpha \pf(0)$,
$(\euh (0), \epth(0), \epfh(0))$ satisfies 
\algns{
&\LRp{2 \mu \e(\euh (0) ), \e(\vb ) } + \LRp{ \epth(0) , \div \vb } = 0 , \\
&\LRp{ \div \euh(0) , \qt } - s_h \LRp{\epth(0), \qt} = \LRp{ \lambda^{-1} (\epti (0) - \alpha \epfi(0)), \qt }
}
for all $\vb \in \Vbh$, $\qt \in \Qth$, therefore 
\algns{
\nor{\euh(0)}_{\Vb} + \nor{\epth(0)}_{\Qt} \lesssim h^k \nor{\pt (0), \pf(0)}_{H^k}, \qquad  k = \min \{ k_{\pt}, k_{\pf} \} .
}
From the boundedness of $s_h(\cdot, \cdot)$, 
\algns{
X(0) \lesssim \nor{\euh(0)}_{\Vb} + \nor{\epth(0)}_{\Qt} + \nor{\epfh(0)}_{L^2_{s_0}} \lesssim h^k \nor{\pt (0), \pf(0)}_{H^k}, 
}
with $k = \min \{ k_{\pt}, k_{\pf} \}$.


{\bf Proof of \eqref{eq:semi-error2}  }: We now estimate $\| \epfh(t) \|_{L^\infty(0,t; H_\kappa^1)}$. For this, we take $\vb = \deuh$ in the time derivative of \eqref{eq:err-eq1}, $\qt = - \depth$ in the time derivative of \eqref{eq:err-eq2}, $\qf = -\depfh$ in \eqref{eq:err-eq3}, and add the equations altogether. Then 
\algn{ \label{eq:err-d-energy-eq}
&\nor{\deuh(t)}_{\Vb}^2 + \nor{\depth(t) - \alpha \depfh(t)}_{0,\lambda^{-1}}^2 + \nor{\depfh(t)}_{0,s_0}^2 + \half \ddt \| \epfh(t) \|_{1,\kappa }^2 \\
\notag &\quad =  -\LRp{ \lambda^{-1} \depti - \alpha \lambda^{-1} \depfi , \depth - \alpha \depfh} + \LRp{s_0 \depfi , \depfh}
}
Integrating it from 0 to $t$ and using Young's inequality, we get 
\mltlns{
\half \| \epfh(t) \|_{1,\kappa}^2 + \int_0^t \LRs{ \nor{\deuh(s)}_{\Vb}^2 + \half \nor{\depth(s) - \alpha \depfh(s)}_{0,\lambda^{-1}}^2 + \half \nor{\depfh(s)}_{0,s_0}^2} \,ds \\
\le \half \| \epfh(0) \|_{1,\kappa}^2 + \half \int_0^t \LRs{ \nor{\depti(s) - \alpha \depfi(s)}_{0,\lambda^{-1}}^2 + \nor{\depfi(s)}_{0,s_0}^2 }\,ds .
}
In particular, 
\algns{
&\| \epfh(t) \|_{1,\kappa} +  \nor{\deuh}_{L^2(0,t; \Vb)} + \nor{\depth - \alpha \depfh}_{L^2(0,t; L^2_{\lambda^{-1}})} + \nor{\depfh}_{L^2(0,t; L^2_{s_0})} \\
&\lesssim \| \epfh(0) \|_{1,\kappa} + \nor{\depti - \alpha \depfi}_{\LtL{2}{\lambda^{-1}}{2}} + \nor{\depfi}_{\LtL{2}{s_0}{2}} \\
&\lesssim \| \epfh(0) \|_{1,\kappa} + h^{k} \norw{\dpt, \dpf}{\LtH{2}{}{k}} , \qquad k = \min\{k_{\pt}, k_{\pf} \} .
}
In this estimate, the implicit constants are uniformly bounded for small $\kappa_{0}$, $\kappa_{1}$, large $\mu$, arbitrarily large $\lambda$, and small or degenerate $s_0$.
\end{proof}
\begin{cor}
Under the same assumptions in Theorem~\ref{thm:eh-estm} and an additional assumption 
\algn{
\label{eq:init-approx3} \norw{\pf(0) - \pfh(0)}{1,\kappa} \lesssim h^{k_{\pf}-1} \norw{\pf(0)}{H^{k_{\pf}}}, 
}
we can show that 
\algns{
&\nor{ \ub - \ubh}_{L^\infty(0,t; \Vb)} + \norw{ \pt - \pth}{L^\infty(0,t; \Qt)} + \nor{ \pf - \pfh}_{L^\infty(0,t; L_{s_0}^2)} \\
\notag & \quad \lesssim h^k \LRp{ \nor{\ub }_{L^\infty(0,t; H^k)} + \nor{\pt}_{W^{1,1}(0,t; H^k)} + \nor{\pf}_{W^{1,1}(0,t; H^k)} }
}
with $k = \min \{k_{\ub}-1, k_{\pt}, k_{\pf} \}$, 
\algns{
& \nor{\pf - \pfh}_{L^2(0,t; H_{\kappa}^1)} \\
\notag & \quad \lesssim h^k \LRp{\nor{\pt(0) }_{H^k} + \nor{\pf(0)}_{H^k} + \nor{\dpt}_{L^1(0,t; H^k)} + \nor{\dpf}_{L^1(0,t; H^{k})} + \nor{\pf}_{L^2(0,t; H^{k+1})}}
}
with $k = \min \{k_{\pt}, k_{\pf}-1 \}$, 
and
\algns{
&\nor{\dub - \dubh}_{L^2(0,t; \Vb)} + \norw{\dpt - \dpth}{L^2(0,t; \Qt)} + \nor{\dpf - \dpfh}_{L^2(0,t; L_{s_0}^2)} \\
\notag & \quad \lesssim \nor{\Pi_h^{\Qf}\pf(0) - \pfh(0)}_{1, \kappa} + h^{k} \norw{\dub, \dpt, \dpf}{\LtH{2}{}{k}} 
}
with $k = \min \{ k_{\ub} - 1, k_{\pt}, k_{\pf} \}$, 
\algns{
&\nor{\pf - \pfh}_{L^\infty(0,t; H_{\kappa}^1)}
\lesssim h^k \LRp{ \nor{\pf }_{L^\infty(0,t; H^{k+1})}+ \norw{\dpt, \dpf}{\LtH{2}{}{k}} } 
}
with $k= \min\{k_{\pf}-1, k_{\pt} \}$ hold.
\end{cor}
\begin{proof}
These assertions can be proved easily from the results in Theorem~\ref{thm:eh-estm} and the triangle inequality, so we omit details.
\end{proof}

\section{Parameter-robust preconditioning}
In this section we discuss preconditioners of the finite element discretizations robust for certain parameter scales. 
In most applications, the parameters $\mu$, $\lambda$, $\kappa$ are in the ranges 
\algn{ \label{eq:param-range}
0 < \kappa_{0}, \kappa_{1} \ll 1 \ll \mu \lesssim \lambda \leq + \infty .
}
It turns out that preconditioners efficient for the model problem with unit parameter values do not perform well for problems with realistic parameter values.
In fact, construction of preconditioners robust for all variations of parameters in \eqref{eq:param-range} is the motivation of \cite{LeeEtAl2017}, and abstract form of parameter-robust block diagonal preconditioners are studied for discretizations with Taylor--Hood and MINI elements. 
Therefore we only focus on preconditioners for discretizations with the two stabilized methods in \eqref{eq:stab-method1} and \eqref{eq:stab-method2}. Following the approach in \cite{LeeEtAl2017}, we first define parameter-dependent discrete norms of $\Vbh$, $\Qth$, $\Qfh$, and show that the stability of the system with the parameter-dependent norms. Then we can derive abstract forms of block diagonal preconditioners based on the parameter-dependent norms. The numerical results we will present in the last section show that performances of algebraic multigrid block diagonal preconditioners based on the abstract forms are robust for parameter scales. 

Before we define parameter-dependent norms, we consider fully discrete schemes of the system to reduce the preconditioning problem. 
In fully discretization scheme of \eqref{eq:weak-upp-semi} with time step size $\lap t >0$, we solve a static system 
\subeqns{eq:weak-upp-static}{
\label{eq:weak-upp-static1} \LRp{2 \mu \e(\ubh ), \e(\vb ) } + \LRp{ \pth , \div \vb } &= (\tilde{\fb}, \vb), \\ 
\label{eq:weak-upp-static2} \LRp{ \div \ubh , \qt } - s_h \LRp{\pth, \qt} - \LRp{ \lambda^{-1} \pth , \qt } - \LRp{\alpha \lambda^{-1} \pfh , \qt}   &= (\tilde{f}, \qt), \\
\label{eq:weak-upp-static3} - \LRp{ \alpha \lambda^{-1} \pth , \qf } - \LRp{\LRp{s_0 + \alpha^2  \lambda^{-1} } \pfh , \qf } - \LRp{\kapb \nabla  \pfh , \nabla \qf }  &= \LRp{ \tilde{g}, \qf } 
}
for all $(\vb, \qt, \qf) \in \Vbh \times \Qth \times \Qfh$ at each time step but $\kapb$ here is  
$\kapb \lap t$ with $\kapb$ in the previous section, and $\tilde{\fb}$, $\tilde{f}$, $\tilde{g}$ are right-hand side terms depending on time discretization schemes. 

Let us define norms of $\Vbh$, $\Qth$, $\Qfh$ as 
\algns{
\nor{\vb}_{\Vbh}^2 &= (2 \mu \e(\vb) , \e(\vb)) , \qquad \nor{\qt}_{\Qth}^2 = ((2\mu)^{-1} \qt, \qt) + s_h(\qt, \qt) , \\
\nor{\qf}_{\Qfh}^2 &= \norw{\qf}{0,s_0}^2 + (\kapb \nabla \qf, \nabla \qf) , 
}
and let $\Xh = \Vbh \times \Qth \times \Qfh$ be the Hilbert space with the norm
\algns{
\norw{(\vb, \qt, \qf)}{\Xh}^2 = \norw{\vb}{\Vbh}^2 + \norw{\qt}{\Qth}^2 + \norw{\Qf}{\Qfh}^2 .
}

We define a linear operator $\mathcal{A}$ from $\Xh$ to its dual space $\Xh^*$ using the left-hand side of \eqref{eq:weak-upp-static} as
\algns{
&\LRa{\mathcal{A} (\ub, \pt, \pf), (\vb, \qt, \qf)}_{(\Xh^*, \Xh)} \\
&\quad = \LRp{2 \mu \e(\ub ), \e(\vb ) } + \LRp{ \pt , \div \vb } + \LRp{ \div \ub , \qt } - s_h \LRp{\pt, \qt} - \LRp{ \lambda^{-1} \pt , \qt } - \LRp{\alpha \lambda^{-1} \pf , \qt} \\
&\qquad - \LRp{ \alpha \lambda^{-1} \pt , \qf } - \LRp{\LRp{s_0 + \alpha^2  \lambda^{-1} } \pf , \qf } - \LRp{\kapb \nabla \pf , \nabla \qf }  
}
for $(\ub, \pt, \pf), (\vb, \qt, \qf) \in \Xh$, where $\LRa{\cdot, \cdot}_{(\Xh^*, \Xh)}$ is the duality pairing of $\Xh$ and $\Xh^*$. 
We claim that $\mc{A}$ is an isomorphism from $\Xh$ to $\Xh^*$ such that $\norw{\mc{A}}{L(\Xh, \Xh^*)}$ and $\norw{\mc{A}^{-1}}{L(\Xh^*, \Xh)}$ are independent of mesh sizes and the parameters in the ranges of \eqref{eq:param-range}. 

\begin{theorem}
There exists ${\beta} >0$, independent of the scales of $\mu$, $\kapb$, $\lambda$ in \eqref{eq:param-range}, and the mesh sizes, such that the following inf-sup condition holds: 
\begin{align*}
\inf_{ (\ub , \pt, \pf) \in \Xh } 
\sup_{(\vb , \qt, \qf) \in \Xh } 
\frac{( {\mathcal{A}} (\ub, \pt, \pf), (\vb, \qt, \qf))_{(\Xh^* , \Xh)} } 
{\| (\ub, \pt, \pf) \|_{\Xh} \| (\vb, \qt, \qf) \|_{\Xh} }  \geq {\beta} .
\end{align*}
\end{theorem}
\begin{proof}
To prove the assertion, for given $(0, 0, 0) \not = (\ub , \pt, \pf) \in \Xh$, we will find $(\vb , \qt, \qf) \in \Xh$ such that 
\begin{align}
\label{eq:inf-sup-sub1} \| (\vb , \qt, \qf) \|_{\Xh} &\leq C \| (\ub , \pt, \pf) \|_{\mathcal{X}}, \\
\label{eq:inf-sup-sub2} ({\mathcal{A}} (\ub , \pt, \pf), (\vb , \qt, \qf))_{(\Xh^*, \Xh )} &\geq C' {\| (\ub , \pt, \pf) \|_{\Xh}^2} ,
\end{align}
with $C, C'>0$ independent of the scales of $\mu$, $\lambda$, $\kapb$, and mesh sizes.

Suppose that $(0, 0, 0) \not = (\ub, \pt, \pf) \in \mathcal{X}$ is given. 

For stabilized methods, there exist $C_1, C_2 >0$ independent of mesh sizes and parameters such that  
\algns{
\sup_{\vb \in \Vbh} \frac{(\div \vb, \qt)}{\norw{\vb}{\Vb}} \ge 2C_1 \norw{\qt}{\Qt} - 2C_2 (s_h(\qt, \qt) )^{\half} \qquad \forall \qt \in \Qth .
}
From this there exists $\wb \in \Vbh$ such that 
\algn{ \label{eq:w-estm}
(\div \wb, \pt) \ge \LRp{ {C_1 \norw{\pt}{\Qt} - C_2 (s_h(\pt, \pt) )^{\half} } }\norw{\wb}{\Vb} .
}
Due to linearity of this inequality in $\wb$ we may rescale $\wb$ so that $\norw{\wb}{\Vb} = \norw{\pt}{\Qt}$. 

To prove \eqref{eq:inf-sup-sub1} and \eqref{eq:inf-sup-sub2}, we set $\vb = \ub + \delta \wb$, $\qt = - \pt$, $\qf = -\pf$ with a constant $\delta>0$ which will be determined later. One can check that 
\begin{align} 
\notag \| (\vb , \qt, \qf) \|_{\Xh} \leq {\sqrt{2(1 + \delta^2 )}} \| (\ub , \pt, \pf) \|_{\Xh} , 
\end{align}
and \eqref{eq:inf-sup-sub1} follows if $\delta$ is independent of the parameters and mesh sizes. 
To establish \eqref{eq:inf-sup-sub2} and determine $\delta$, we use the previously chosen $\vb$, $\qt$, $\qf$, and \eqref{eq:w-estm} to have 
\begin{align} 
\notag &\LRa{ {\mathcal{A}} (\ub , \pt, \pf), (\vb , \qt, \qf)}_{(\Xh^* , \Xh)} \\
\label{eq:inf-sup-bilinear} &= \| \ub \|_{\Vbh}^2 + \delta (2 \mu \e(\ub), \e(\wb)) + \delta (\div \wb, \pt) + s_h(\pt, \pt)  \\
\notag & \quad + ( \lambda^{-1} \pt, \pt) + ((s_0 + \alpha^2 \lambda^{-1}) \pf, \pf) + 2 (\alpha \lambda^{-1}\pt, \pf)) + (\kapb \nabla \pf, \nabla \pf)  
\end{align}
By Young's inequality and the fact $\norw{\wb}{\Vb} = \norw{\pt}{\Qt}$, we also have
\begin{align*}
& \delta (2 \mu \e(\bs{u}), \e(\bs{w})) \leq \frac{\delta \theta}{2} \| \bs{u} \|_{\Vb}^2 + \frac{\delta}{2 \theta} \| \wb \|_{\Vb}^2 \le \frac{\delta \theta}{2} \| \ub \|_{\Vb}^2 + \frac{\delta}{2 \theta} \| \pt \|_{\Qt}^2 
\quad \forall \theta >0.
\end{align*}
By \eqref{eq:w-estm} and Young's inequality, 
\algns{
\delta (\div \ub, \pt) &\ge \delta \LRp{C_1 \norw{\pt}{\Qt} - C_2 (s_h(\pt, \pt) )^{\half} } \norw{\wb}{\Vb} \\
&\ge \delta C_1 \norw{\pt}{\Qt}^2 - \delta C_2 \LRp{\frac{\eta}{2} s_h(\pt, \pt) + \frac{1}{2\eta} \norw{\pt}{\Qt}^2 }
}
for any $\eta >0$. From these we can get 
\algns{ 
&\LRa{ {\mathcal{A}} (\ub , \pt, \pf), (\vb , \qt, \qf)}_{(\Xh^* , \Xh)} \\
&= \LRp{ 1 - \frac{\delta \theta}{2} } \nor{\ub}_{\Vb}^2 + \delta \LRp{C_1 - \frac{1}{2\theta} - \frac{C_2}{2 \eta} } \nor{\pt}_{\Qt}^2 + \LRp{1 - \delta \frac{C_2 \eta}{2} } s_h(\pt, \pt) \\
&\quad + \nor{ \pt - \alpha \pf}_{\lambda^{-1}}^2 + \nor{\pf}_{s_0}^2 + \nor{ \pf}_{1,\kappa}^2 .
}
We now set 
\algns{
\theta = \frac{2}{C_1},\qquad \eta = \frac{2 C_2}{C_1}, \qquad \delta = \min \LRb{\frac{C_1}{2}, \frac{C_1}{2 C_2^2} }, 
}
and get 
\mltlns{ 
\LRa{ {\mathcal{A}} (\ub , \pt, \pf), (\vb , \qt, \qf)}_{(\Xh^* , \Xh)} \\
\ge \half \| \ub \|_{\Vb}^2 + \half s_h(\pt, \pt) + \frac{\delta C_1}{2} \norw{\pt}{\Qt}^2 + \norw{\pf}{0,s_0}^2 + \nor{\pf}_{1,\kappa}^2  .
}
Since $C_1$, $C_2$ are independent of parameters and mesh sizes, so is $\delta$, and therefore \eqref{eq:inf-sup-sub1} and \eqref{eq:inf-sup-sub2}
are proved.
\end{proof}

\begin{table}[h]
	\begin{center}
		\begin{tabular}{c | c c | c c | c c | c c} 
			\multirow{3}{*}{$N$} & \multicolumn{2}{c}{$\nor{\pt - \pth}_{L^2}$}& \multicolumn{2}{c}{$\nor{\pf - \pfh}_{L^2}$}& \multicolumn{2}{c}{$\nor{\ub - \ubh}_{H^1}$}& \multicolumn{2}{c}{$\nor{\pf - \pfh}_{1,\kappa}$}\\ 
\cline{2-9}& error & rate & error & rate & error & rate & error & rate \\
\hline
8 & 4.342e-02  & $-$  & 3.527e-03  & $-$  & 5.725e-02  & $-$  & 1.127e-01  & $-$ \\
16 & 1.071e-02  & 2.02  & 8.826e-04  & 2.00  & 1.424e-02  & 2.01  & 5.642e-02  & 1.00 \\
32 & 2.669e-03  & 2.00  & 2.207e-04  & 2.00  & 3.559e-03  & 2.00  & 2.822e-02  & 1.00 \\
64 & 6.668e-04  & 2.00  & 5.519e-05  & 2.00  & 8.897e-04  & 2.00  & 1.411e-02  & 1.00 \\
128 & 1.667e-04  & 2.00  & 1.380e-05  & 2.00  & 2.225e-04  & 2.00  & 7.056e-03  & 1.00 \\
		\hline
		\end{tabular} 
		\caption{Errors and convergence rates with the lowest order Taylor--Hood finite elements} 
		\label{TH-conv} 
	\end{center} 
\end{table}

\begin{table}[h]
	\begin{center}
		\begin{tabular}{c | c c | c c | c c | c c} 
			\multirow{3}{*}{$N$} & \multicolumn{2}{c}{$\nor{\pt - \pth}_{L^2}$}& \multicolumn{2}{c}{$\nor{\pf - \pfh}_{L^2}$}& \multicolumn{2}{c}{$\nor{\ub - \ubh}_{H^1}$}& \multicolumn{2}{c}{$\nor{\pf - \pfh}_{1,\kappa}$}\\ 
\cline{2-9}& error & rate & error & rate & error & rate & error & rate \\
\hline
8 & 6.024e+00  & $-$  & 3.549e-03  & $-$  & 7.017e+00  & $-$  & 1.139e-01  & $-$ \\
16 & 3.748e+00  & 0.68  & 1.409e-03  & 1.33  & 4.438e+00  & 0.66  & 5.723e-02  & 0.99 \\
32 & 1.519e+00  & 1.30  & 5.280e-04  & 1.42  & 1.793e+00  & 1.31  & 2.843e-02  & 1.01 \\
64 & 4.642e-01  & 1.71  & 1.643e-04  & 1.68  & 5.463e-01  & 1.71  & 1.415e-02  & 1.01 \\
128 & 1.259e-01  & 1.88  & 4.586e-05  & 1.84  & 1.568e-01  & 1.80  & 7.061e-03  & 1.00 \\
		\hline
		\end{tabular} 
		\caption{Errors and convergence rates with the Brezzi--Pitk\"{a}ranta stabilized method} 
		\label{BP-conv} 
	\end{center} 
\end{table}

\begin{table}[h]
	\begin{center}
		\begin{tabular}{c | c c | c c | c c | c c} 
			\multirow{3}{*}{$N$} & \multicolumn{2}{c}{$\nor{\pt - \pth}_{L^2}$}& \multicolumn{2}{c}{$\nor{\pf - \pfh}_{L^2}$}& \multicolumn{2}{c}{$\nor{\ub - \ubh}_{H^1}$}& \multicolumn{2}{c}{$\nor{\pf - \pfh}_{1,\kappa}$}\\ 
\cline{2-9}& error & rate & error & rate & error & rate & error & rate \\
\hline
8 & 5.051e+00  & $-$  & 1.149e-02  & $-$  & 5.816e+00  & $-$  & 1.223e-01  & $-$ \\
16 & 2.604e+00  & 0.96  & 2.726e-03  & 2.08  & 2.906e+00  & 1.00  & 5.772e-02  & 1.08 \\
32 & 1.349e+00  & 0.95  & 6.126e-04  & 2.15  & 1.402e+00  & 1.05  & 2.835e-02  & 1.03 \\
64 & 6.832e-01  & 0.98  & 1.490e-04  & 2.04  & 6.762e-01  & 1.05  & 1.413e-02  & 1.01 \\
128 & 3.296e-01  & 1.05  & 3.787e-05  & 1.98  & 3.082e-01  & 1.13  & 7.058e-03  & 1.00 \\
		\hline
		\end{tabular} 
		\caption{Errors and convergence rates with the $\mc{P}_1$-$\mc{P}_0$ stabilized method} 
		\label{CGDG-conv} 
	\end{center} 
\end{table}

The above stability in the parameter-dependent norm $\mc{X}_h$ suggests an abstract form of preconditioner 
\algn{ \label{eq:precond}
\bs{P} = \pmat{P_{\ub} & & \\ & P_{\pt} & \\ & & P_{\pf} }
}
with $P_{\ub}$, $P_{\pt}$, $P_{\pf}$ which are (approximate) inverses of the maps
\algns{
\ub \mapsto -\div (2 \mu \e (\ub)), \qquad \pt \mapsto (1/\mu) \pt, \qquad \pf \mapsto (s_0 + \alpha^2 \lambda^{-1})\pf - \div (\kapb \nabla \pf) .
}

\begin{table}[th]
	\begin{center}
		\begin{tabular}{>{\small}c | >{\small}c | >{\small}c || >{\small}c >{\small}c >{\small}c >{\small}c }  
\hline 
& & & \multicolumn{4}{>{\small}c}{$\kappa$} \\ 
$N$ & $\mu$ & $\lambda/\mu$ & $10^0$ & $10^{-3}$ & $10^{-6}$ & $10^{-9}$ \\ 
\hline \hline 
\multirow{9}{*} {$16$} & \multirow{3}{*} {$10^0$} & $10^0$ & $44 \; (0.19)$& $56 \; (0.24)$& $65 \; (0.27)$& $65 \; (0.27)$\\ 
& & $10^3$ & $59 \; (0.27)$& $55 \; (0.24)$& $70 \; (0.29)$& $67 \; (0.27)$\\ 
& & $10^6$& $60 \; (0.25)$& $54 \; (0.22)$& $42 \; (0.18)$& $53 \; (0.22)$\\ 
\cline{3-7} 
& \multirow{3}{*} {$10^3$} & $10^0$ & $41 \; (0.17)$& $41 \; (0.17)$& $54 \; (0.22)$& $62 \; (0.25)$\\ 
& & $10^3$& $59 \; (0.25)$& $59 \; (0.25)$& $52 \; (0.22)$& $67 \; (0.29)$\\ 
& & $10^6$ & $59 \; (0.25)$& $59 \; (0.25)$& $51 \; (0.22)$& $38 \; (0.16)$\\ 
\cline{3-7} 
& \multirow{3}{*} {$10^6$} & $10^0$ & $41 \; (0.17)$& $41 \; (0.17)$& $41 \; (0.17)$& $54 \; (0.23)$\\ 
& & $10^3$ & $59 \; (0.25)$& $59 \; (0.25)$& $59 \; (0.25)$& $52 \; (0.22)$\\ 
& & $10^6$ & $59 \; (0.25)$& $59 \; (0.25)$& $59 \; (0.25)$& $51 \; (0.22)$\\ 
\cline{3-7} 
\cline{2-7} 
\multirow{9}{*} {$32$} & \multirow{3}{*} {$10^0$} & $10^0$ & $46 \; (0.41)$& $55 \; (0.53)$& $67 \; (0.63)$& $67 \; (0.61)$\\ 
& & $10^3$ & $61 \; (0.56)$& $56 \; (0.52)$& $73 \; (0.66)$& $69 \; (0.61)$\\ 
& & $10^6$ & $62 \; (0.58)$& $55 \; (0.54)$& $42 \; (0.40)$& $54 \; (0.51)$\\ 
\cline{3-7} 
& \multirow{3}{*} {$10^3$} & $10^0$ & $42 \; (0.41)$& $42 \; (0.41)$& $52 \; (0.50)$& $63 \; (0.58)$\\ 
& & $10^3$ & $61 \; (0.58)$& $61 \; (0.58)$& $52 \; (0.48)$& $68 \; (0.60)$\\ 
& & $10^6$ & $61 \; (0.59)$& $61 \; (0.60)$& $52 \; (0.49)$& $38 \; (0.36)$\\ 
\cline{3-7} 
& \multirow{3}{*} {$10^6$} & $10^0$ & $42 \; (0.41)$& $42 \; (0.41)$& $42 \; (0.41)$& $52 \; (0.47)$\\ 
& & $10^3$ & $60 \; (0.58)$& $61 \; (0.59)$& $61 \; (0.62)$& $52 \; (0.53)$\\ 
& & $10^6$ & $60 \; (0.52)$& $61 \; (0.59)$& $61 \; (0.63)$& $52 \; (0.53)$\\ 
\cline{3-7} 
\cline{2-7} 
\multirow{9}{*} {$64$} & \multirow{3}{*} {$10^0$} & $10^0$ & $46 \; (1.47)$& $55 \; (1.91)$& $67 \; (1.98)$& $68 \; (2.16)$\\ 
& & $10^3$ & $61 \; (1.88)$& $56 \; (1.72)$& $72 \; (2.26)$& $69 \; (2.37)$\\ 
& & $10^6$ & $61 \; (2.19)$& $56 \; (1.88)$& $42 \; (1.40)$& $53 \; (1.93)$\\ 
\cline{3-7} 
& \multirow{3}{*} {$10^3$} & $10^0$ & $42 \; (1.51)$& $42 \; (1.37)$& $50 \; (1.52)$& $63 \; (1.86)$\\ 
& & $10^3$ & $61 \; (1.85)$& $61 \; (1.97)$& $52 \; (1.84)$& $66 \; (2.05)$\\ 
& & $10^6$& $61 \; (2.05)$& $61 \; (2.14)$& $52 \; (1.82)$& $37 \; (1.34)$\\ 
\cline{3-7} 
& \multirow{3}{*} {$10^6$} & $10^0$ & $42 \; (1.55)$& $42 \; (1.60)$& $42 \; (1.43)$& $50 \; (1.86)$\\ 
& & $10^3$ & $60 \; (2.10)$& $61 \; (2.19)$& $61 \; (2.12)$& $52 \; (1.84)$\\ 
& & $10^6$ & $60 \; (2.26)$& $61 \; (1.86)$& $61 \; (1.88)$& $52 \; (1.63)$\\ 
\cline{3-7} 
\cline{2-7} 
\multirow{9}{*} {$128$} & \multirow{3}{*} {$10^0$} & $10^0$ & $46 \; (7.30)$& $55 \; (8.39)$& $67 \; (10.19)$& $70 \; (10.53)$\\ 
& & $10^3$ & $63 \; (10.27)$& $56 \; (8.78)$& $72 \; (11.37)$& $68 \; (10.63)$\\ 
& & $10^6$ & $63 \; (10.50)$& $56 \; (9.09)$& $41 \; (6.36)$& $52 \; (7.50)$\\ 
\cline{3-7} 
& \multirow{3}{*} {$10^3$} & $10^0$ & $42 \; (6.53)$& $42 \; (7.12)$& $50 \; (7.80)$& $63 \; (9.11)$\\ 
& & $10^3$ & $62 \; (9.13)$& $62 \; (10.03)$& $52 \; (8.69)$& $66 \; (10.63)$\\ 
& & $10^6$ & $62 \; (9.87)$& $62 \; (10.10)$& $52 \; (8.29)$& $37 \; (5.94)$\\ 
\cline{3-7} 
& \multirow{3}{*} {$10^6$} & $10^0$ & $42 \; (6.72)$& $42 \; (6.54)$& $42 \; (5.86)$& $50 \; (6.95)$\\ 
& & $10^3$ & $62 \; (8.48)$& $62 \; (8.49)$& $61 \; (8.50)$& $52 \; (7.18)$\\ 
& & $10^6$ & $62 \; (8.45)$& $62 \; (8.32)$& $62 \; (8.34)$& $52 \; (7.02)$\\ 
\cline{3-7} 
\cline{2-7} 
\multirow{9}{*} {$256$} & \multirow{3}{*} {$10^3$} & $10^0$ & $46 \; (25.82)$& $54 \; (30.11)$& $65 \; (35.70)$& $69 \; (38.01)$\\ 
& & $10^3$ & $62 \; (35.41)$& $56 \; (31.50)$& $71 \; (40.29)$& $70 \; (38.55)$\\ 
& & $10^6$ & $63 \; (35.49)$& $55 \; (31.34)$& $41 \; (23.18)$& $51 \; (29.03)$\\ 
\cline{3-7} 
& \multirow{3}{*} {$10^3$} & $10^3$ & $43 \; (24.18)$& $43 \; (24.33)$& $51 \; (28.79)$& $59 \; (32.19)$\\ 
& & $10^3$ & $60 \; (34.44)$& $60 \; (33.83)$& $51 \; (28.70)$& $64 \; (35.27)$\\ 
& & $10^6$ & $60 \; (34.06)$& $60 \; (33.44)$& $50 \; (28.59)$& $37 \; (21.02)$\\ 
\cline{3-7} 
& \multirow{3}{*} {$10^6$} & $10^0$ & $42 \; (24.00)$& $42 \; (23.66)$& $42 \; (24.02)$& $50 \; (28.52)$\\ 
& & $10^3$ & $60 \; (33.83)$& $60 \; (33.90)$& $60 \; (33.89)$& $51 \; (28.97)$\\ 
& & $10^6$ & $60 \; (33.80)$& $60 \; (33.91)$& $60 \; (34.40)$& $50 \; (28.71)$\\ 
\cline{3-7} 
\hline 
		\end{tabular} 
		\caption{Number of iterations and wall-clock time for one solve with the Taylor--Hood element} 
		\label{table:TH-precond} 
	\end{center} 
\end{table} 

\begin{table}[ht]
	\begin{center}
		\begin{tabular}{>{\small}c | >{\small}c | >{\small}c || >{\small}c >{\small}c >{\small}c >{\small}c }  
\hline 
& & & \multicolumn{4}{>{\small}c}{$\kappa$} \\ 
$N$ & $\mu$ & $\lambda/\mu$ & $10^0$ & $10^{-3}$ & $10^{-6}$ & $10^{-9}$ \\ 
\hline \hline 
\multirow{9}{*} {$16$} & \multirow{3}{*} {$10^0$} & $10^0$ & $21 \; (0.07)$& $32 \; (0.10)$& $30 \; (0.09)$& $30 \; (0.09)$\\ 
& & $10^3$ & $22 \; (0.07)$& $20 \; (0.06)$& $28 \; (0.09)$& $26 \; (0.08)$\\ 
& & $10^6$ & $22 \; (0.07)$& $20 \; (0.06)$& $16 \; (0.05)$& $23 \; (0.07)$\\ 
\cline{3-7} 
& \multirow{3}{*} {$10^3$} & $10^0$ & $17 \; (0.06)$& $20 \; (0.06)$& $30 \; (0.09)$& $28 \; (0.09)$\\ 
& & $10^3$ & $21 \; (0.07)$& $21 \; (0.07)$& $19 \; (0.06)$& $26 \; (0.08)$\\ 
& & $10^6$ & $21 \; (0.07)$& $21 \; (0.07)$& $18 \; (0.06)$& $14 \; (0.05)$\\ 
\cline{3-7} 
& \multirow{3}{*} {$10^9$} & $10^0$ & $17 \; (0.06)$& $17 \; (0.06)$& $19 \; (0.06)$& $30 \; (0.10)$\\ 
& & $10^3$ & $21 \; (0.07)$& $21 \; (0.07)$& $21 \; (0.07)$& $19 \; (0.06)$\\ 
& & $10^6$ & $21 \; (0.07)$& $21 \; (0.07)$& $21 \; (0.07)$& $18 \; (0.06)$\\ 
\cline{3-7} 
\cline{2-7} 
\multirow{9}{*} {$32$} & \multirow{3}{*} {$1$} & $1$ & $24 \; (0.11)$& $37 \; (0.17)$& $35 \; (0.16)$& $35 \; (0.16)$\\ 
& & $10^3$ & $29 \; (0.13)$& $25 \; (0.12)$& $33 \; (0.16)$& $31 \; (0.14)$\\ 
& & $10^6$ & $29 \; (0.14)$& $25 \; (0.12)$& $20 \; (0.09)$& $26 \; (0.12)$\\ 
\cline{3-7} 
& \multirow{3}{*} {$10^3$} & $10^0$ & $19 \; (0.09)$& $22 \; (0.11)$& $35 \; (0.16)$& $32 \; (0.15)$\\ 
& & $10^3$ & $26 \; (0.13)$& $26 \; (0.13)$& $23 \; (0.11)$& $30 \; (0.15)$\\ 
& & $10^6$ & $26 \; (0.12)$& $26 \; (0.12)$& $23 \; (0.11)$& $18 \; (0.09)$\\ 
\cline{3-7} 
& \multirow{3}{*} {$10^6$} & $10^0$ & $19 \; (0.09)$& $19 \; (0.09)$& $22 \; (0.11)$& $35 \; (0.16)$\\ 
& & $10^3$ & $26 \; (0.12)$& $26 \; (0.12)$& $26 \; (0.12)$& $23 \; (0.11)$\\ 
& & $10^6$ & $26 \; (0.13)$& $26 \; (0.13)$& $26 \; (0.12)$& $23 \; (0.11)$\\ 
\cline{3-7} 
\cline{2-7} 
\multirow{9}{*} {$64$} & \multirow{3}{*} {$10^0$} & $10^0$ & $26 \; (0.31)$& $40 \; (0.48)$& $38 \; (0.41)$& $38 \; (0.39)$\\ 
& & $10^3$ & $35 \; (0.39)$& $31 \; (0.34)$& $38 \; (0.43)$& $36 \; (0.37)$\\ 
& & $10^6$ & $35 \; (0.40)$& $30 \; (0.35)$& $24 \; (0.28)$& $29 \; (0.36)$\\ 
\cline{3-7} 
& \multirow{3}{*} {$10^3$} & $10^0$ & $21 \; (0.25)$& $24 \; (0.29)$& $37 \; (0.44)$& $36 \; (0.39)$\\ 
& & $10^3$ & $32 \; (0.38)$& $32 \; (0.35)$& $28 \; (0.31)$& $35 \; (0.41)$\\ 
& & $10^6$ & $32 \; (0.38)$& $32 \; (0.35)$& $27 \; (0.30)$& $20 \; (0.23)$\\ 
\cline{3-7} 
& \multirow{3}{*} {$10^6$} & $10^0$ & $20 \; (0.23)$& $21 \; (0.24)$& $24 \; (0.27)$& $37 \; (0.44)$\\ 
& & $10^3$ & $32 \; (0.38)$& $32 \; (0.38)$& $32 \; (0.37)$& $27 \; (0.32)$\\ 
& & $10^6$ & $32 \; (0.38)$& $32 \; (0.38)$& $32 \; (0.38)$& $27 \; (0.32)$\\ 
\cline{3-7} 
\cline{2-7} 
\multirow{9}{*} {$128$} & \multirow{3}{*} {$10^0$} & $10^0$ & $28 \; (1.08)$& $44 \; (1.71)$& $43 \; (1.54)$& $42 \; (1.51)$\\ 
& & $10^3$ & $39 \; (1.63)$& $33 \; (1.33)$& $42 \; (1.75)$& $40 \; (1.51)$\\ 
& & $10^6$ & $39 \; (1.64)$& $33 \; (1.36)$& $26 \; (1.07)$& $31 \; (1.31)$\\ 
\cline{3-7} 
& \multirow{3}{*} {$10^3$} & $10^0$ & $23 \; (0.98)$& $25 \; (1.05)$& $40 \; (1.70)$& $39 \; (1.54)$\\ 
& & $10^3$ & $37 \; (1.56)$& $37 \; (1.57)$& $32 \; (1.36)$& $40 \; (1.74)$\\ 
& & $10^6$ & $37 \; (1.60)$& $37 \; (1.57)$& $32 \; (1.33)$& $23 \; (1.00)$\\ 
\cline{3-7} 
& \multirow{3}{*} {$10^6$} & $10^0$ & $22 \; (0.95)$& $23 \; (0.99)$& $25 \; (1.07)$& $40 \; (1.72)$\\ 
& & $10^3$ & $36 \; (1.59)$& $36 \; (1.55)$& $36 \; (1.53)$& $31 \; (1.36)$\\ 
& & $10^6$ & $36 \; (1.51)$& $36 \; (1.53)$& $36 \; (1.54)$& $31 \; (1.28)$\\ 
\cline{3-7} 
\cline{2-7} 
\multirow{9}{*} {$256$} & \multirow{3}{*} {$10^0$} & $10^0$ & $28 \; (4.83)$& $46 \; (8.01)$& $46 \; (8.10)$& $44 \; (6.71)$\\ 
& & $10^3$ & $42 \; (8.16)$& $35 \; (7.12)$& $45 \; (8.65)$& $44 \; (7.78)$\\ 
& & $10^6$ & $42 \; (7.55)$& $35 \; (6.38)$& $26 \; (5.02)$& $33 \; (6.27)$\\ 
\cline{3-7} 
& \multirow{3}{*} {$10^3$} & $10^3$ & $25 \; (4.80)$& $26 \; (5.19)$& $42 \; (7.19)$& $43 \; (6.62)$\\ 
& & $10^3$ & $41 \; (7.53)$& $41 \; (7.82)$& $34 \; (6.34)$& $42 \; (8.35)$\\ 
& & $10^6$ & $41 \; (7.47)$& $41 \; (7.16)$& $34 \; (6.63)$& $26 \; (5.22)$\\ 
\cline{3-7} 
& \multirow{3}{*} {$10^6$} & $10^0$ & $22 \; (4.54)$& $24 \; (4.73)$& $26 \; (4.97)$& $41 \; (7.65)$\\ 
& & $10^3$ & $40 \; (7.56)$& $40 \; (8.14)$& $40 \; (7.57)$& $33 \; (6.48)$\\ 
& & $10^6$ & $40 \; (7.80)$& $40 \; (8.04)$& $40 \; (7.89)$& $33 \; (6.70)$\\ 
\cline{3-7} 
\hline 
		\end{tabular} 
		\caption{Number of iterations and wall-clock time for one solve with the Brezzi--Pitk\"{a}ranta stabilized method} 
		\label{table:BP-precond} 
	\end{center} 
\end{table} 

\begin{table}[ht]
	\begin{center}
		\begin{tabular}{>{\small}c | >{\small}c | >{\small}c || >{\small}c >{\small}c >{\small}c >{\small}c }  
\hline 
& & & \multicolumn{4}{>{\small}c}{$\kappa$} \\ 
$N$ & $\mu$ & $\lambda$ & $10^0$ & $10^{-3}$ & $10^{-6}$ & $10^{-9}$ \\ 
\hline \hline 
\multirow{9}{*} {$16$} & \multirow{3}{*} {$10^0$} & $10^0$ & $24 \; (0.08)$& $38 \; (0.12)$& $37 \; (0.12)$& $37 \; (0.12)$\\ 
& & $10^3$ & $33 \; (0.11)$& $30 \; (0.10)$& $39 \; (0.13)$& $36 \; (0.12)$\\ 
& & $10^6$ & $33 \; (0.11)$& $29 \; (0.10)$& $22 \; (0.07)$& $30 \; (0.10)$\\ 
\cline{3-7} 
& \multirow{3}{*} {$10^3$} & $10^0$ & $20 \; (0.07)$& $23 \; (0.08)$& $37 \; (0.12)$& $35 \; (0.11)$\\ 
& & $10^3$ & $31 \; (0.10)$& $32 \; (0.11)$& $29 \; (0.09)$& $38 \; (0.12)$\\ 
& & $10^6$ & $31 \; (0.10)$& $32 \; (0.10)$& $28 \; (0.09)$& $22 \; (0.08)$\\ 
\cline{3-7} 
& \multirow{3}{*} {$10^6$} & $10^0$ & $19 \; (0.06)$& $20 \; (0.07)$& $23 \; (0.08)$& $37 \; (0.12)$\\ 
& & $10^3$ & $31 \; (0.10)$& $31 \; (0.10)$& $32 \; (0.10)$& $29 \; (0.09)$\\ 
& & $10^6$ & $31 \; (0.10)$& $31 \; (0.10)$& $32 \; (0.10)$& $28 \; (0.09)$\\ 
\cline{3-7} 
\cline{2-7} 
\multirow{9}{*} {$32$} & \multirow{3}{*} {$10^0$} & $10^0$ & $26 \; (0.13)$& $42 \; (0.20)$& $41 \; (0.20)$& $41 \; (0.20)$\\ 
& & $10^3$ & $38 \; (0.19)$& $34 \; (0.17)$& $43 \; (0.21)$& $40 \; (0.19)$\\ 
& & $10^6$ & $38 \; (0.19)$& $33 \; (0.16)$& $26 \; (0.13)$& $33 \; (0.16)$\\ 
\cline{3-7} 
& \multirow{3}{*} {$10^3$} & $10^0$ & $21 \; (0.11)$& $24 \; (0.12)$& $40 \; (0.20)$& $39 \; (0.19)$\\ 
& & $10^3$ & $36 \; (0.18)$& $36 \; (0.19)$& $32 \; (0.17)$& $42 \; (0.22)$\\ 
& & $10^6$ & $36 \; (0.18)$& $36 \; (0.19)$& $32 \; (0.16)$& $24 \; (0.13)$\\ 
\cline{3-7} 
& \multirow{3}{*} {$10^6$} & $10^0$ & $21 \; (0.11)$& $21 \; (0.11)$& $24 \; (0.13)$& $40 \; (0.20)$\\ 
& & $10^3$ & $35 \; (0.18)$& $36 \; (0.19)$& $36 \; (0.19)$& $32 \; (0.17)$\\ 
& & $10^6$ & $35 \; (0.18)$& $36 \; (0.19)$& $37 \; (0.19)$& $32 \; (0.16)$\\ 
\cline{3-7} 
\cline{2-7} 
\multirow{9}{*} {$64$} & \multirow{3}{*} {$10^0$} & $10^0$ & $26 \; (0.33)$& $43 \; (0.54)$& $43 \; (0.50)$& $43 \; (0.49)$\\ 
& & $10^3$ & $39 \; (0.48)$& $35 \; (0.43)$& $46 \; (0.56)$& $41 \; (0.48)$\\ 
& & $10^6$ & $39 \; (0.47)$& $35 \; (0.43)$& $27 \; (0.34)$& $35 \; (0.43)$\\ 
\cline{3-7} 
& \multirow{3}{*} {$10^3$} & $10^0$ & $22 \; (0.28)$& $25 \; (0.31)$& $41 \; (0.52)$& $40 \; (0.46)$\\ 
& & $10^3$ & $37 \; (0.46)$& $38 \; (0.47)$& $34 \; (0.41)$& $43 \; (0.55)$\\ 
& & $10^6$ & $37 \; (0.45)$& $38 \; (0.48)$& $34 \; (0.42)$& $25 \; (0.33)$\\ 
\cline{3-7} 
& \multirow{3}{*} {$10^6$} & $10^0$ & $21 \; (0.28)$& $22 \; (0.28)$& $24 \; (0.32)$& $41 \; (0.48)$\\ 
& & $10^3$ & $36 \; (0.42)$& $37 \; (0.46)$& $38 \; (0.50)$& $33 \; (0.43)$\\ 
& & $10^6$ & $36 \; (0.44)$& $37 \; (0.43)$& $38 \; (0.45)$& $33 \; (0.39)$\\ 
\cline{3-7} 
\cline{2-7} 
\multirow{9}{*} {$128$} & \multirow{3}{*} {$10^0$} & $10^0$ & $28 \; (1.25)$& $46 \; (1.98)$& $46 \; (1.86)$& $46 \; (1.77)$\\ 
& & $10^3$ & $42 \; (1.82)$& $38 \; (1.67)$& $49 \; (2.13)$& $43 \; (1.62)$\\ 
& & $10^6$ & $42 \; (1.67)$& $37 \; (1.48)$& $28 \; (1.12)$& $37 \; (1.58)$\\ 
\cline{3-7} 
& \multirow{3}{*} {$10^3$} & $10^0$ & $24 \; (1.04)$& $27 \; (1.06)$& $43 \; (1.70)$& $44 \; (1.65)$\\ 
& & $10^3$ & $40 \; (1.75)$& $40 \; (1.78)$& $37 \; (1.61)$& $46 \; (2.02)$\\ 
& & $10^6$ & $40 \; (1.58)$& $40 \; (1.71)$& $35 \; (1.53)$& $28 \; (1.18)$\\ 
\cline{3-7} 
& \multirow{3}{*} {$10^6$} & $10^0$ & $23 \; (1.03)$& $24 \; (1.07)$& $26 \; (1.18)$& $43 \; (1.80)$\\ 
& & $10^3$ & $38 \; (1.64)$& $39 \; (1.56)$& $40 \; (1.83)$& $36 \; (1.68)$\\ 
& & $10^6$ & $38 \; (1.82)$& $39 \; (1.72)$& $40 \; (1.72)$& $35 \; (1.47)$\\ 
\cline{3-7} 
\cline{2-7} 
\multirow{9}{*} {$256$} & \multirow{3}{*} {$10^0$} & $10^0$ & $29 \; (5.25)$& $46 \; (9.32)$& $49 \; (7.97)$& $48 \; (7.97)$\\ 
& & $10^3$ & $50 \; (9.03)$& $44 \; (8.35)$& $57 \; (10.94)$& $54 \; (9.90)$\\ 
& & $10^6$ & $50 \; (9.34)$& $44 \; (8.31)$& $33 \; (6.76)$& $43 \; (8.90)$\\ 
\cline{3-7} 
& \multirow{3}{*} {$10^3$} & $10^0$ & $25 \; (4.85)$& $28 \; (5.29)$& $45 \; (9.21)$& $47 \; (8.14)$\\ 
& & $10^3$ & $48 \; (9.19)$& $48 \; (9.40)$& $43 \; (7.84)$& $55 \; (10.93)$\\ 
& & $10^6$ & $48 \; (9.76)$& $48 \; (9.52)$& $43 \; (8.55)$& $32 \; (6.44)$\\ 
\cline{3-7} 
& \multirow{3}{*} {$10^6$} & $10^0$ & $24 \; (4.82)$& $24 \; (4.84)$& $27 \; (5.40)$& $44 \; (9.02)$\\ 
& & $10^3$ & $46 \; (8.77)$& $47 \; (8.23)$& $47 \; (8.00)$& $42 \; (7.22)$\\ 
& & $10^6$ & $46 \; (7.82)$& $47 \; (7.97)$& $48 \; (8.07)$& $42 \; (7.11)$\\ 
\cline{3-7} 
\hline 
		\end{tabular} 
		\caption{Number of iterations and wall-clock time for one solve with the $\mc{P}_1$--$\mc{P}_0$ stabilized method} 
		\label{table:CGDG-precond} 
	\end{center} 
\end{table} 

\section{Numerical results}

In this section we present the results of numerical experiments. All numerical experiments are performed with FEniCS version 2017.2.0. 

In the first numerical experiments, we show convergence of finite finite element methods. 
The computational $\Omega$ is the unit square $[0,1]\times [0,1]$ and is divided into $N \times N$ uniform squares, i.e., $h = 1/N$,
and then each squares are divided into to two triangles to obtain the triangulation $\mathcal{T}_h$. 
To illustrate convergence of errors, we consider a manufactured solution of the problem with 
\algns{
\ub = \pmat{\sin(\pi x) \sin(1 + t) \\ \sin y \sin t }, \qquad p = x^2 y^2 \cos t
}
and parameters $\mu = 10$, $\lambda = 15$, $\alpha = 1$, $s_0 = 1$, $\kappa = 1$.
For boundary conditions we impose Dirichlet boundary conditions of $\ub$ on $\Gamma_d := \{0\} \times [0,1] \cup \{1\} \times [0,1]$
and of $\pf$ on $\Gamma_p := \pd \Omega$.

We consider the lowest order Taylor--Hood element, the Brezzi--Pitk\"{a}ranta stabilized method (cf. \eqref{eq:stab-method1}), and the $\mc{P}_1$--$\mc{P}_0$ stabilized method (cf. \eqref{eq:stab-method2}). We use the backward Euler time discretization with time step $\lap t = h^2$ and the errors are computed at $t = 0.5$.
Convergence rates of errors for mesh refinements are given in Tables~\ref{TH-conv}--\ref{CGDG-conv}. 

Although parameter-robust preconditioning for mixed methods are already studied in \cite{LeeEtAl2017}, we show the results of mixed method and stabilized methods for comparision.
To construct preconditioners based on \eqref{eq:precond} for mixed methods, we use the algebraic multigrid method for the blocks of $\ub$ and $\pf$ but use the Jacobi preconditioner for the block of $\pt$ as in \cite{LeeEtAl2017}:  
\algns{
\pmat{\text{AMG} (A_{\ub}) & & \\ & \text{Jacobi}(A_{\pt}) & \\ & & \text{AMG}(A_{\pf}) } 
}
where $A_{\ub}$, $A_{\pt}$, $A_{\pf}$ are matrices obtained from the bilinear forms
\algns{
(2 \mu \e(\ub), \e(\vb)), \quad ((2\mu)^{-1} \pt, \qt) , \quad ((s_0 + \alpha^2 \lambda^{-1}) \pf, \qf) + (\kapb \nabla \pf, \nabla \qf) .
}
For stabilized methods our preconditioners have the form 
\algns{
\pmat{\text{AMG} (A_{\ub}) & & \\ & \text{AMG}(A_{\pt}) & \\ & & \text{AMG}(A_{\pf}) }
}
where $A_{\ub}$, $A_{\pt}$, $A_{\pf}$ are matrices obtained by 
\algns{
(2 \mu \e(\ub), \e(\vb)), \quad ((2\mu)^{-1} \pt, \qt) + s_h(\pt, \qt), \quad ((s_0 + \alpha^2 \lambda^{-1}) \pf, \qf) + (\kapb \nabla \pf, \nabla \qf) 
}
for each stabilized method, and MinRes algorithm is used for iterative solvers. For algebraic multigrid methods we use the algebraic multigrid package Hypre AMG.

To test robustness of these preconditioners for mesh refinements, and parameter values, we consider the cases with meshes $N = 16, 32, 64,128,256$, $\mu = 1, 10^3, 10^6$, $\lambda/\mu = 1, 10^3, 10^6$, and scalar $\kappa = 1, 10^{-3}, 10^{-6}, 10^{-9}$. At each case, we only test the static problem with randomly generated right-hand side vectors, and measured number of iterations with relative tolerance $10^{-6}$, and measured the wall-clock time for one solve by averaging 10 solves with different right-hand side vectors. The results are given in Tables~\ref{table:TH-precond}--\ref{table:CGDG-precond}. 
One can see that the numbers of iteration in Tables~\ref{table:BP-precond}-\ref{table:CGDG-precond} are quite robust for different parameter values and mesh refinements. 
In the results, the stabilized methods have significantly less computational times for same meshes, so they can be advantageous to accelerate simulations but the price to pay is the low accuracy of stabilized methods as we have seen before.

\section{Conclusion}
In this paper we studied the three-field formulation of the Biot model which has the displacement, the total pressure, and the pore pressure as unknowns.
We first carried out a comprehensive investigation of the a priori estimate of the continuous problem. Then we studied finite element discretization with parameter-robust stability, and parameter-robust preconditioning of the discretizations.
For finite element discretizations we considered standard mixed finite element as well as stabilized methods for the Stokes equations, and complete error estimates of semidiscrete solutions of the Biot model are proved. For parameter-robust preconditioning, we showed parameter-robust stability of the system and derived an abstract form of robust preconditioners. The theoretical results are illustrated with numerical experiments.

%
%
%
%
%
%
%

\providecommand{\bysame}{\leavevmode\hbox to3em{\hrulefill}\thinspace}
\providecommand{\MR}{\relax\ifhmode\unskip\space\fi MR }
\providecommand{\MRhref}[2]{%
  \href{http://www.ams.org/mathscinet-getitem?mr=#1}{#2}
}
\providecommand{\href}[2]{#2}

\bibliographystyle{amsplain}
\vspace{.125in}

\end{document}